\documentclass[a4paper,11pt]{amsart}

\usepackage[top=30truemm,bottom=25truemm,left=35truemm,right=35truemm]{geometry}

\usepackage{amsmath}
\usepackage{amssymb}
\usepackage{amsthm}
\usepackage{graphics}
\usepackage[abbrev,alphabetic]{amsrefs}
\usepackage{amscd}
\usepackage[dvipdfmx]{graphicx}
\usepackage{color}
\usepackage{ulem}
\usepackage{enumitem}

\usepackage[all,cmtip]{xy}
\usepackage{xypic}

\title[A characterization of virtually free actions via arc spaces]
{A characterization of virtually free actions via arc spaces and its application to the lower semi-continuity conjecture}

\author{Yusuke Nakamura}
\address{Graduate School of Mathematics, Nagoya University, Furo-cho, Chikusa-ku, Nagoya, 464-8602, Japan.}
\email{y.nakamura@math.nagoya-u.ac.jp}
\urladdr{https://sites.google.com/site/ynakamuraagmath/}

\author{Kohsuke Shibata}
\address{School of Engineering, Tokyo Denki University, Adachi-ku, Tokyo 120-8551, Japan.}
\email{shibata.kohsuke@mail.dendai.ac.jp}



\newtheorem{thm}{Theorem}[section]
\newtheorem{lem}[thm]{Lemma}

\newtheorem{cor}[thm]{Corollary}
\newtheorem{prop}[thm]{Proposition}

\newtheorem{claim}[thm]{Claim}

\theoremstyle{definition}
\newtheorem{defi}[thm]{Definition}
\newtheorem{eg}[thm]{Example}
\newtheorem{conj}[thm]{Conjecture}

\theoremstyle{remark}
\newtheorem{rmk}[thm]{Remark}

\newtheorem*{ackn}{Acknowledgements}
\newtheorem{setup}[thm]{Setup}

\begin{document}

\maketitle

\begin{abstract}
In this paper, we study the precise inversion of adjunction (PIA) conjecture and the lower semi-continuity (LSC) conjecture for hyperquotient singularities. 
Previously known results for these conjectures in this setting required the singularity to be klt, and without this assumption, a counterexample to the PIA conjecture is known to exist.

To resolve this obstacle, we introduce a localized notion of virtually free actions and characterize it via the arc spaces of quotient varieties. 
Utilizing this characterization, we establish a necessary and sufficient condition for the PIA conjecture to hold for arbitrary hyperquotient singularities, thereby clarifying the mechanism of the counterexample. 
Furthermore, as an application of this insight, we unconditionally establish the LSC conjecture for arbitrary hyperquotient singularities.
\end{abstract}

\section{Introduction}

The minimal log discrepancy is a fundamental invariant of singularities in birational geometry. 
In \cite{Amb99}, Ambro proposed the lower semi-continuity (LSC) conjecture for minimal log discrepancies. 
As shown by Shokurov \cite{Sho04}, the LSC conjecture, together with the ascending chain condition (ACC) conjecture, implies the termination of flips. 
In this paper, we focus on the LSC conjecture and the related precise inversion of adjunction (PIA) conjecture. 
Throughout this paper, we work over an algebraically closed field $k$ of characteristic zero.

The LSC conjecture asserts the lower semi-continuity of the minimal log discrepancy as the closed point varies:

\begin{conj}[LSC conjecture, \cite{Amb99}]\label{conj:LSC}
Let $(X, \mathfrak{a})$ be a log pair, and let $|X|_{\rm cl}$ denote the set of all closed points of $X$ with the Zariski topology. 
Then the function 
\[
|X|_{\rm cl} \to \mathbb{R}_{\ge 0} \cup \{ - \infty \}; \quad x \mapsto \operatorname{mld}_x (X,\mathfrak{a})
\]
is lower semi-continuous. 
\end{conj}

The LSC conjecture has been confirmed in the following settings:
\begin{enumerate}
\item[(\ref{conj:LSC}.1)]
Ambro \cite{Amb99} proved the conjecture for $\dim X = 3$. 

\item[(\ref{conj:LSC}.2)]
Ein, Musta\c{t}\u{a}, and Yasuda \cite{EMY03} proved it when $X$ is smooth. 

\item[(\ref{conj:LSC}.3)]
Ein and Musta\c{t}\u{a} \cite{EM04} generalized the result (\ref{conj:LSC}.2) to normal local complete intersection varieties. 

\item[(\ref{conj:LSC}.4)]
The first author \cite{Nak16} proved the conjecture when $X$ has quotient singularities. 

\item[(\ref{conj:LSC}.5)]
In \cite{NS22, NS25}, the authors generalized the results (\ref{conj:LSC}.3) and (\ref{conj:LSC}.4) to hyperquotient klt singularities. 
More generally, they proved the conjecture for a finite group quotient of a complete intersection variety defined by invariant equations, provided that the variety $X$ is klt 
(the case of semi-invariant equations is treated in \cite{NS26}). 
\end{enumerate}

The following PIA conjecture is useful for replacing the variety under consideration with a better ambient space when studying minimal log discrepancies.

\begin{conj}[PIA conjecture, {\cite{92}*{17.3.1}}]\label{conj:PIA}
Let $(X, \mathfrak{a})$ be a log pair, and let $D$ be a normal prime Cartier divisor on $X$. 
Let $x \in D$ be a closed point. Suppose that $D$ is not contained in the cosupport of the $\mathbb{R}$-ideal sheaf $\mathfrak{a}$. 
Then, we have
\[
\operatorname{mld}_x \bigl( X, \mathfrak{a} \mathcal{O}_X(-D) \bigr) = \operatorname{mld}_x (D, \mathfrak{a} \mathcal{O}_D). 
\]
\end{conj}

The PIA conjecture has been confirmed in the following settings:
\begin{enumerate}
\item[(\ref{conj:PIA}.1)]
The conjecture is known to hold for $\dim X = 3$ (cf.\ \cite{NS4}*{Proposition 6.1}).

\item[(\ref{conj:PIA}.2)]
Ein, Musta\c{t}\u{a}, and Yasuda \cite{EMY03} proved the conjecture when $X$ is smooth. 

\item[(\ref{conj:PIA}.3)]
Ein and Musta\c{t}\u{a} \cite{EM04} generalized the result (\ref{conj:PIA}.2) to normal local complete intersection varieties. 

\item[(\ref{conj:PIA}.4)]
In \cite{NS22, NS25}, the authors proved the conjecture when $X$ has quotient singularities and $D$ is klt at $x$. 
More generally, they established it for $X$ a finite group quotient of a complete intersection variety defined by invariant equations, provided that both $X$ and $D$ are klt at $x$ (the case of semi-invariant equations is treated in \cite{NS26}). 
\end{enumerate}

\noindent
Despite these affirmative results, the conjecture is not true in full generality: 

\begin{enumerate}
\item[(\ref{conj:PIA}.5)]
In \cite{NS4}, the authors provided a counterexample $(X, D)$ to the conjecture, where $X$ is a $5$-dimensional variety with quotient singularities and $D$ is log canonical but not klt. 
\end{enumerate}

\noindent
This counterexample (\ref{conj:PIA}.5) demonstrates that the klt assumption in the result (\ref{conj:PIA}.4) is indeed indispensable. 

The LSC conjecture and the PIA conjecture are intertwined as follows. 
In general, assuming the PIA conjecture for $X$ and $D$, the LSC conjecture for an ambient space $X$ implies the LSC conjecture for $D$. 
Indeed, the result (\ref{conj:LSC}.3) in \cite{EM04} for the LSC conjecture was proved by reducing it to the smooth case (\ref{conj:LSC}.2) via the result (\ref{conj:PIA}.3) for the PIA conjecture. 
Similarly, the result (\ref{conj:LSC}.5) in \cite{NS22, NS25, NS26} was established by reducing it to the case of quotient singularities (\ref{conj:LSC}.4) using the result (\ref{conj:PIA}.4). 
The klt assumption in the result (\ref{conj:LSC}.5) was indispensable because it was required to enable this reduction.

However, as demonstrated by the counterexample (\ref{conj:PIA}.5), the PIA conjecture fails for hyperquotient singularities that are not klt, meaning we can no longer rely on this naive reduction strategy. 
Furthermore, the authors have shown in \cite{NS4}*{Example 5.1} that there also exist counterexamples to the LSC conjecture for families in the setting of hyperquotient singularities without the klt assumption. 
Consequently, it is highly natural to suspect that non-klt hyperquotient singularities could serve as counterexamples to the original LSC conjecture (Conjecture \ref{conj:LSC}) as well. 

The primary objective of this paper is to uncover the mechanism behind the counterexample to the PIA conjecture for hyperquotient singularities, thereby clarifying the fundamental role of the klt condition. 
Furthermore, contrary to the pessimistic expectation mentioned above, we utilize this insight to unconditionally prove the LSC conjecture for hyperquotient singularities, completely removing the previously required klt assumption from the result (\ref{conj:LSC}.5).

\begin{thm}[$=$ Theorem \ref{thm:LSC}]\label{thm:LSC_intro}
Suppose that a finite subgroup $G$ acts on a smooth variety $A$. 
Let $A' := A / G$ be the quotient variety. 
Let $X$ be a normal subvariety of $A'$ of codimension $c$ 
such that $X$ is locally defined by $c$ equations in $A'$. 
Then, the LSC conjecture (Conjecture \ref{conj:LSC}) is true for $X$.
\end{thm}

To establish Theorem \ref{thm:LSC_intro} without relying on the reduction via the PIA conjecture, a new approach is required. Specifically, we achieve this by characterizing virtually free actions in terms of their arc spaces. 

The concept of virtually free actions was recently introduced by the authors in \cite{NS4} to construct the counterexample to the PIA conjecture. 
Let $G$ be a finite group acting on a normal variety $X$. 
The $G$-action on $X$ is said to be \textit{virtually free} if, for every $G$-equivariant proper birational morphism $Y \to X$ from a normal variety $Y$, the $G$-action on $Y$ is free in codimension one (see Definition \ref{defi:vf}). 
As noted in Remark \ref{rmk:vf}, this concept is an intermediate notion between a free action and an action that is free in codimension one. 
Furthermore, when $X$ is klt, the property of being virtually free is equivalent to being free (see Theorem \ref{thm:kltcase}).

In this paper, we also introduce a localized version of this notion.
For a closed point $x \in X$ and an element $\gamma \in G$, we say that the $\gamma$-action on $X$ is \textit{virtually free at $x$} if, for every $\gamma$-equivariant proper birational morphism $f \colon Y \to X$ from a normal variety $Y$, the $\gamma$-fixed locus on $Y$ intersects the fiber $f^{-1}(x)$ in a subset of codimension at least two in $Y$ (see Definition \ref{defi:vf_local}).
Let $\operatorname{NVF}(X,x)$ denote the set of elements $\gamma \in G$ such that the $\gamma$-action on $X$ is not virtually free at $x$.

A key consequence of this localized notion is that the elements $\gamma \in G$ whose action is virtually free at $x$ do not contribute to the drop in the minimal log discrepancy. 

\begin{prop}[$=$ Proposition \ref{prop:eq_local}]\label{prop:eq_local_intro}
Suppose that a finite group $G$ acts on a normal $\mathbb{Q}$-Gorenstein variety $X$, and that the action is free in codimension one.
Let $\mathfrak{a}$ be a non-zero $\mathbb{R}$-ideal sheaf on $X/G$. 
Let $x \in X$ be a closed point, and let $x' \in X/G$ be its image. 
Then, we have 
\[
\operatorname{mld}_{x'}(X/G, \mathfrak{a}) = 
\min _{\gamma \in \operatorname{NVF}(X,x)}
\operatorname{mld}_{x' _{\gamma}} \bigl( X/\langle \gamma \rangle , \mathfrak{a} \mathcal{O}_{X/\langle \gamma \rangle} \bigr), 
\]
where $x'_{\gamma}$ denotes the image of $x$ in $X/\langle \gamma \rangle$. 
\end{prop}

\noindent
In \cite{NS22}*{Remark 4.13}, it is proved that 
\[
\operatorname{mld}_{x'}(X/G, \mathfrak{a}) = \min _{\gamma \in G} \operatorname{mld}_{x' _{\gamma}} \bigl( X/\langle \gamma \rangle , \mathfrak{a} \mathcal{O}_{X/\langle \gamma \rangle} \bigr). 
\]
Proposition \ref{prop:eq_local_intro} asserts the stronger fact that this minimum is always attained at some element $\gamma \in \operatorname{NVF}(X,x)$ (see Remark \ref{rmk:eq_local}).

To present our main result on the PIA conjecture, we introduce the following simplified setting of quotient singularities (see Setup \ref{setup:nonlin} for the general setting).

\begin{setup}\label{setup:intro}
Suppose that a finite group $G$ acts on a smooth variety $X$, and that the action is free in codimension one. 
Let $X' := X / G$ be the quotient variety. 
Let $x \in X$ be a $G$-fixed closed point, and let $x' \in X'$ be its image. 
Let $D'$ be a normal prime Cartier divisor on $X'$ passing through $x'$, and let $D$ be its inverse image in $X$. 
\end{setup}

Under this setup, we establish a necessary and sufficient condition for the PIA conjecture to hold, thereby clarifying the mechanism of the counterexample.

\begin{thm}[$\subset$ Theorem \ref{thm:PIA_equiv}]\label{thm:PIA_equiv_intro}
In the setting of Setup \ref{setup:intro}, let $\mathfrak{b}$ be an $\mathbb{R}$-ideal sheaf on $X'$. 
Suppose that $D'$ is not contained in the cosupport of $\mathfrak{b}$. 
Then the following conditions are equivalent: 
\begin{enumerate}
\item[(i)]
The minimum 
\[
\min _{\gamma \in \operatorname{NVF}(X,x)} \operatorname{mld}_{x'_{\gamma}} \bigl( X/\langle \gamma \rangle, \mathfrak{b} \mathcal{O}_{X/\langle \gamma \rangle} (-D_{\gamma}) \bigr)
\]
is attained at some $\gamma \in \operatorname{NVF}(D,x)$, where $x'_{\gamma} \in X/\langle \gamma \rangle$ denotes the image of $x$, and $D_{\gamma} := D/\langle \gamma \rangle$. 

\item[(ii)]
The PIA conjecture holds; that is, 
\[
\operatorname{mld}_{x'} \bigl( X', \mathfrak{b} \mathcal{O}_{X'} (-D') \bigr) = 
\operatorname{mld}_{x'} (D', \mathfrak{b} \mathcal{O}_{D'}). 
\]
\end{enumerate}
\end{thm}

As a direct consequence of this characterization, we obtain the following practical sufficient condition for the PIA conjecture. This occurs exactly when condition (i) is automatically satisfied.

\begin{cor}[$\subset$ Corollary \ref{cor:PIA_general}]\label{cor:PIA_general_intro}
In the setting of Setup \ref{setup:intro}, let $\mathfrak{b}$ be an $\mathbb{R}$-ideal sheaf on $X'$. 
Suppose that $D'$ is not contained in the cosupport of $\mathfrak{b}$. 
If $\operatorname{NVF}(D,x) = G$, then it follows that
\[
\operatorname{mld}_{x'} \bigl( X', \mathfrak{b} \mathcal{O}_{X'} (-D') \bigr) = 
\operatorname{mld}_{x'} (D', \mathfrak{b} \mathcal{O}_{D'}). 
\]
\end{cor}

\noindent
The LSC conjecture for $D'$ under Setup \ref{setup:intro}, which is a special case of Theorem \ref{thm:LSC_intro}, then follows from Proposition \ref{prop:eq_local_intro} and Corollary \ref{cor:PIA_general_intro}, together with the semi-continuity of the set $\operatorname{NVF}(D,x)$ (see Lemma \ref{lem:vf_properties}(2)).

The proof of Theorem \ref{thm:PIA_equiv_intro} relies on the theory of arc spaces for quotient varieties, originally developed by Denef and Loeser \cite{DL02} and later generalized by the authors in \cite{NS22}. 
In Denef and Loeser's framework, the arc space $X'_{\infty}$ of the quotient variety $X' = X/G$ is investigated by decomposing it into subsets ${X'}^{(\gamma)}_{\infty} \subset X'_{\infty}$ associated with each element $\gamma \in G$ (see Definition \ref{defi:dl1}). 
However, when dealing with non-klt varieties, a critical bottleneck arises: as pointed out in \cite{NS22}*{Introduction}, the subset ${X'}^{(\gamma)}_{\infty}$ can become a \textit{thin set} of $X'_{\infty}$. 
In such cases, it becomes too small to carry meaningful geometric information. 
This phenomenon is precisely the reason why the klt assumption was indispensable in the previous result (\ref{conj:PIA}.4), and why the counterexample (\ref{conj:PIA}.5) occurs without it.

To overcome this obstacle, we establish a connection between this arc-theoretic behavior and our localized notion of virtually free actions. 

\begin{thm}[$\subset$ Theorem \ref{thm:vf_thin_local}]\label{thm:thin_eq_intro}
Suppose that a finite group $G$ acts on a normal variety $X$. 
Let $X' = X/G$ be the quotient variety. 
Let $x \in X$ be a $G$-fixed closed point, and let $x' \in X'$ be its image.
For an element $\gamma \in G$, the following conditions are equivalent:
\begin{enumerate}
\item The $\gamma$-action on $X$ is not virtually free at $x$ (i.e., $\gamma \in \operatorname{NVF}(X,x)$).
\item The local arc space ${X'}^{(\gamma)}_{\infty, x'}$ is not a thin set of $X'_{\infty}$ (see Remark \ref{rmk:identity_thin}).
\end{enumerate}
Here, ${X'}^{(\gamma)}_{\infty, x'}$ denotes the set of arcs in ${X'}^{(\gamma)}_{\infty}$ centered at $x'$ (see Definition \ref{defi:arc_local}).
\end{thm}

\noindent
This theorem is utilized in Theorem \ref{thm:5.2}, which is the key to the proof of Theorem \ref{thm:PIA_equiv_intro}.

In the approach of \cite{DL02}, the subset ${X'}^{(\gamma)}_{\infty}$ is investigated by defining a certain $k[t]$-scheme $X^{(\gamma)}$ and considering its arc space (Greenberg scheme) $X^{(\gamma)}_{\infty}$ as a $k[t]$-scheme (see Section \ref{section:dl2}). 
In this paper, we introduce a new $k[t]$-scheme, denoted by $X^{[\gamma]}$, to prove the above theorem. 
By definition, $X^{[\gamma]}$ is the quotient of $X \times \operatorname{Spec} k[t^{1/d'}]$ by the $\langle \gamma \rangle$-action, where $d'$ is the order of $\gamma$ (see Definition \ref{defi:twisted_quotient}). 
Both $X^{(\gamma)}$ and $X^{[\gamma]}$ can be viewed as ``branched versions'' of the twisted variety associated with the $G$-action on $X$. 
Indeed, they are isomorphic over the open subset $t \neq 0$ (see Remarks \ref{rmk:twist} and \ref{rmk:iso_out0}), and there is a natural bijection between their sets of $K$-arcs (Proposition \ref{prop:()to[]}).

The primary advantage of $X^{[\gamma]}$ lies in its functorial construction. 
This functoriality makes $X^{[\gamma]}$ highly compatible with the concept of virtually free actions, which is defined via $G$-equivariant resolutions. 
On the other hand, the classical construction $X^{(\gamma)}$ is defined only for linear actions and lacks this functoriality. 
However, it is important to note that $X^{(\gamma)}$ cannot be entirely replaced by $X^{[\gamma]}$ in the arguments of \cite{DL02} and \cite{NS22}. 
The crucial advantage of $X^{(\gamma)}$ is that the number of its defining equations exactly coincides with that of $X$ (see Example \ref{eg:shiki}). 
This property is well-suited for the computation of the canonical divisors and the estimation of codimensions of arc spaces. 
In fact, the proof of the previous result (\ref{conj:PIA}.4) in \cite{NS22} heavily relied on this advantage of $X^{(\gamma)}$. 
In this paper, by bridging the geometric functoriality of $X^{[\gamma]}$ and the algebraic tractability of $X^{(\gamma)}$, we complete the proof of Theorem \ref{thm:PIA_equiv_intro}, as carried out in Theorems \ref{thm:thin_eq_intro}, \ref{thm:5.2}, and \ref{thm:PIA}.

We emphasize here that Corollary \ref{cor:PIA_general_intro} (and consequently Theorem \ref{thm:PIA_equiv_intro}) can be seen as a generalization of the previous result (\ref{conj:PIA}.4), which strictly required the klt assumption.
This is because $\operatorname{NVF}(D,x) = G$ holds if $D$ is klt at $x$ (see Theorem \ref{thm:klt_local}). 
It is worth noting that the proof of Theorem \ref{thm:klt_local} relies on the following deep geometric facts:  
\begin{itemize}
\item 
Takayama's theorem \cite{Tak03}*{Theorem 1.1} on the simple connectedness of the fibers of a resolution of a klt point; or

\item
Hacon and McKernan's result \cite{HM07} on the rational chain connectedness of such fibers, combined with the theorem of Graber, Harris, and Starr \cite{GHS03} asserting the existence of a section for morphisms with rationally connected fibers.
\end{itemize}

The paper is organized as follows. 
In Section \ref{section:pre}, we review basic definitions and properties concerning log pairs, arc spaces, and virtually free actions. 
In Section \ref{section:vfc}, we introduce the $k[t]$-scheme $X^{[\gamma]}$ and characterize virtually free actions in terms of its arc space. 
In Sections \ref{section:dl1} and \ref{section:dl2}, we study the arc spaces of quotient varieties for general and linear actions, respectively, relating our framework to the theory established by Denef and Loeser in \cite{DL02}. 
In Section \ref{section:vf_local}, we introduce the localized notions of virtually free actions and local arc spaces, and prove Proposition \ref{prop:eq_local} ($=$ Proposition \ref{prop:eq_local_intro}). 
In Section \ref{section:arc_local}, we investigate further properties of local arc spaces. In particular, we establish Theorem \ref{thm:5.2}, which plays a crucial role in removing the klt assumption. 
In Section \ref{section:PIA}, we prove Theorem \ref{thm:PIA_equiv} (the full version of Theorem \ref{thm:PIA_equiv_intro}) as our main result on the PIA conjecture, establishing necessary and sufficient conditions for hyperquotient singularities. We then deduce Corollary \ref{cor:PIA_general} (the full version of Corollary \ref{cor:PIA_general_intro}).
In Section \ref{section:cex}, we re-examine the counterexample given in \cite{NS4} from the viewpoint of this necessary and sufficient condition, clarifying the mechanism behind the failure of the PIA conjecture. 
Finally, in Section \ref{section:LSC}, we prove Theorem \ref{thm:LSC} ($=$ Theorem \ref{thm:LSC_intro}), our main result regarding the LSC conjecture for hyperquotient singularities.

\begin{ackn} 
The first author is partially supported by Inamori Foundation and by JSPS KAKENHI No.\ 22K13888. 
The second author is partially supported by JSPS KAKENHI No.\ 23K12958.
\end{ackn}

\section*{Notation}

\begin{itemize}
\item 
$k$ denotes an algebraically closed field of characteristic zero. 

\item
A variety over $k$ is an integral separated scheme of finite type over $k$. 

\item
An action of a finite group $G$ on a variety $X$ is called \textit{free in codimension one} 
if it is free on the set of closed points, $|X|_{\rm cl}$, away from a closed subset of codimension at least two.

\item
When considering the quotient $X/G$ of a variety $X$ by a finite group $G$, we always assume that $X/G$ is a variety. 
Note that this assumption is satisfied, for instance, if $X$ is quasi-projective over $k$.
\end{itemize}

\section{Preliminaries}\label{section:pre}
Following \cite{NS22}, we review some definitions and facts regarding log pairs and arc spaces. 

\subsection{Log pairs}
A \textit{log pair} $(X, \mathfrak{a})$ consists of a normal $\mathbb{Q}$-Gorenstein variety $X$ over $k$ and 
an $\mathbb{R}$-ideal sheaf $\mathfrak{a}$ on $X$. 
Here, an $\mathbb{R}$-\textit{ideal sheaf} $\mathfrak{a}$ on $X$ is a formal product 
$\mathfrak{a} = \prod _{i = 1} ^s \mathfrak{a}_i ^{\tau_i}$, where $\mathfrak{a}_1, \ldots, \mathfrak{a}_s$ are 
non-zero coherent ideal sheaves on $X$ 
and $\tau _1, \ldots , \tau _s$ are positive real numbers. 
For a morphism $Y \to X$ and an $\mathbb{R}$-ideal sheaf $\mathfrak{a} = \prod _{i = 1} ^s \mathfrak{a}_i ^{\tau _i}$ on $X$, 
we denote by $\mathfrak{a} \mathcal{O}_Y$ the $\mathbb{R}$-ideal sheaf $\prod _{i = 1} ^s (\mathfrak{a}_i \mathcal{O}_Y)  ^{\tau _i}$ on $Y$. 

Let $\bigl( X, \mathfrak{a} = \prod _{i = 1} ^s \mathfrak{a}_i ^{\tau _i} \bigr)$ be a log pair. 
Let $f\colon X' \to X$ be a proper birational morphism from a normal variety $X'$, and let $E$ be a prime divisor on $X'$. 
We denote the relative canonical divisor by $K_{X'/X} := K_{X'} - f^* K_X$. 
Then the \textit{log discrepancy} of $(X, \mathfrak{a})$ at $E$ is defined as 
\[
a_E(X, \mathfrak{a}) := 1 + \operatorname{ord}_E (K_{X'/X}) - \operatorname{ord}_E \mathfrak{a}, 
\]
where we define $\operatorname{ord}_E \mathfrak{a} := \sum _{i=1} ^s \tau _i \operatorname{ord}_E \mathfrak{a}_i$. 
The image $f(E)$ is called the \textit{center} of $E$ on $X$, and we denote it by $c_X(E)$. 
For a closed point $x \in X$, we define the \textit{minimal log discrepancy} at $x$ as 
\[
\operatorname{mld}_x (X, \mathfrak{a}) := \inf _{c_X(E) = \{ x \}} a_E (X, \mathfrak{a})
\]
if $\dim X \ge 2$, where the infimum is taken over all prime divisors $E$ over $X$ with center $c_X(E) = \{ x \}$. 
It is known that $\operatorname{mld}_x (X, \mathfrak{a}) \in \mathbb{R}_{\ge 0} \cup \{ - \infty \}$ in this case (cf.\ \cite{KM98}*{Corollary 2.31}). 
When $\dim X = 1$, we define $\operatorname{mld}_x (X, \mathfrak{a}) := \inf _{c_X(E) = \{ x \}} a_E (X, \mathfrak{a})$ 
if the infimum is non-negative, and we define $\operatorname{mld}_x (X, \mathfrak{a}) := - \infty$ otherwise. 

Let $D$ be a $\mathbb{Q}$-Cartier divisor on $X$. 
Take $\ell \in \mathbb{Z} _{>0}$ such that $\ell D$ is Cartier. 
Then we define 
\[
a_E (X, D) := a_E \bigl( X, (\mathcal{O}_X(- \ell D)) ^{\frac{1}{\ell}} \bigr), \quad 
\operatorname{mld}_x (X, D) := \operatorname{mld}_x \bigl( X, (\mathcal{O}_X(- \ell D)) ^{\frac{1}{\ell}} \bigr). 
\]
Note that these values do not depend on the choice of $\ell$. 

When $\mathfrak{a} = \mathcal{O}_X$, we write $a_E (X) := a_E (X, \mathfrak{a})$ and $\operatorname{mld}_x (X) := \operatorname{mld}_x (X, \mathfrak{a})$ for simplicity.

\subsection{Arc spaces of $k$-schemes}\label{subsection:arc}
In this subsection, we briefly review the definitions and some properties of jet schemes and arc spaces. The reader is referred to \cite{EM09} for details.

Let $X$ be a scheme over $k$. 
Let $({\sf Sch}/k)$ be the category of $k$-schemes and $({\sf Sets})$ the category of sets.
Define a contravariant functor $F_{m} \colon ({\sf Sch}/k) \to ({\sf Sets})$ by 
\[
 F_{m}(Y) = \operatorname{Hom} _{k}\left( Y\times_{\operatorname{Spec} k} \operatorname{Spec} k[t]/(t^{m+1}), X \right).
\]
This functor $F_{m}$ is represented by a scheme $X_m$ over $k$, and 
the scheme $X_m$ is called the $m$-th \textit{jet scheme} of $X$.
For $m \ge n \ge 0$, the canonical surjective homomorphism $k[t]/(t^{m+1}) \to k[t]/(t^{n+1})$ induces a morphism $\pi_{mn} \colon X_m \to X_n$.
The projective limit 
\[
X_\infty := \varprojlim_{m} X_m
\]
exists, together with natural projections $\psi_{m} \colon X_{\infty} \to X_m$. The scheme $X_{\infty}$ is called the \textit{arc space} of $X$. 
For any field extension $K$ of $k$, there is a natural bijection
\[
\operatorname{Hom} _{k}(\operatorname{Spec} K, X_{\infty}) \simeq \operatorname{Hom} _{k}(\operatorname{Spec} K[[t]], X).
\]

For $m\in\mathbb Z_{\ge 0}\cup\{\infty\}$ and a morphism $f \colon Y \to X$ of schemes of finite type over $k$,
we denote by $f_m \colon Y_m \to X_m$ the morphism induced by $f$.

A subset $C \subset X_{\infty}$ is called a \textit{cylinder} if $C = \psi_{m} ^{-1}(S)$ holds for some $m \ge 0$ and 
a constructible subset $S \subset X_m$. A typical example of a cylinder appearing in this paper is the \textit{contact locus} $\operatorname{Cont}^{\geq m}(\mathfrak{a})$, which is defined as follows. 

\begin{defi}\label{defi:ord_cont}
\begin{enumerate}
\item 
For an arc $\alpha \in X_{\infty}$ represented by a morphism $\alpha \colon \operatorname{Spec} K[[t]] \to X$ (where $K$ is a field extension of $k$) and an ideal sheaf $\mathfrak{a} \subset \mathcal{O}_X$, 
the \textit{order} of $\mathfrak{a}$ measured by $\alpha$ is defined as
\[
\operatorname{ord}_{\alpha} (\mathfrak{a}) = \sup \{ r \in \mathbb{Z}_{\geq 0} \mid \alpha^*(\mathfrak{a}) \subset (t^r) \} \in \mathbb{Z}_{\ge 0} \cup \{ \infty \}, 
\]
where $\alpha^*(\mathfrak{a})$ denotes the ideal of $K[[t]]$ generated by the image of $\mathfrak{a}$.

\item 
For $m \in \mathbb{Z}_{\ge 0}$, the \textit{contact locus} $\operatorname{Cont}^{\geq m}(\mathfrak{a}) \subset X_{\infty}$ is defined as
\[
\operatorname{Cont}^{\geq m}(\mathfrak{a}) = \{ \alpha \in X_\infty \mid \operatorname{ord}_{\alpha}(\mathfrak{a}) \geq m\}.
\]
\end{enumerate}
\end{defi}

\begin{defi}\label{defi:thin_k}
Let $X$ be a variety over $k$. A subset $A \subset X_{\infty}$ is said to be \textit{thin} if $A \subset Z_{\infty}$ for some proper closed subscheme $Z$ of $X$. 
\end{defi}

\begin{rmk}\label{rmk:thin_cylinder_k}
For a cylinder $A \subset X_{\infty}$, $A$ is thin if and only if $A \subset (X_{\rm sing})_{\infty}$, where $X_{\rm sing}$ denotes the singular locus of $X$ (see \cite{EM09}*{Lemma 5.1}). 
\end{rmk}

\subsection{Arc spaces of $k[t]$-schemes}\label{subsection:arc2}
In this subsection, we briefly discuss arc spaces of $k[t]$-schemes. 
We refer the reader to \cite{DL02} and \cite{NS22} for further details.

Let $X$ be a scheme over $k[t]$. 
For a non-negative integer $m$, we define a contravariant functor $F^X_{m} \colon ({\sf Sch}/k) \to ({\sf Sets})$ by 
\[
F^X _{m}(Y) = \operatorname{Hom} _{k[t]} \left( Y \times_{\operatorname{Spec} k} \operatorname{Spec} k[t]/(t^{m+1}), X \right).
\]
This functor $F^X_{m}$ is represented by a scheme $X_m$ over $k$. 
We retain the notation $X_{\infty}$, $\psi_m$, and $\pi_{mn}$ used in the previous setting: 
\[
X_\infty := \mathop{\varprojlim}\limits_{m} X_m, \qquad \psi_{m} \colon X_{\infty} \to X_m, \qquad \pi_{mn} \colon X_m \to X_n. 
\]
A subset $C \subset X_{\infty}$ is called a \textit{cylinder} if it is of the form $C = \psi_{m}^{-1}(S)$ for some integer $m \ge 0$ and a constructible subset $S \subset X_m$. 
The \textit{contact locus} $\operatorname{Cont}^{\geq m}(\mathfrak{a})$ is defined in the same way as in Definition \ref{defi:ord_cont}.

\begin{rmk}\label{rmk:bc}
Let $X$ be a scheme of finite type over $k$, and let $X' = X \times_{\operatorname{Spec} k} \operatorname{Spec} k[t]$ be its base change to $k[t]$. 
Then there is a natural isomorphism $X_m \simeq X'_m$ for any $m \in \mathbb{Z}_{\ge 0} \cup \{ \infty \}$. 
Here, $X_m$ denotes the jet or arc scheme over $k$ as defined in Subsection \ref{subsection:arc}, while $X'_m$ denotes the corresponding scheme over $k[t]$ as defined in this subsection. 
Therefore, the theory of arc spaces of $k[t]$-schemes can be viewed as a generalization of that for $k$-schemes. 
\end{rmk}

\begin{rmk}\label{rmk:GB}
The schemes $X_m$ and $X_{\infty}$ coincide with the Greenberg schemes $\operatorname{Gr}_m(\mathfrak{X})$ and $\operatorname{Gr}_{\infty}(\mathfrak{X})$ (as defined in \cite{Seb04} and \cite{CLNS}) of the formal scheme $\mathfrak{X}$ over $k[[t]]$ associated to the $k[t]$-scheme $X$. We refer the reader to \cite{NS22}*{Remark 2.14} for details.
\end{rmk}

\begin{rmk}\label{rmk:dom}
Let $X$ be a scheme of finite type over $k[t]$. 
Let $X^{(1)}, \ldots, X^{(\ell)}$ be the irreducible components of $X$ that dominate $\operatorname{Spec} k[t]$. 
Then the arc space $X_{\infty}$ decomposes into the union of the arc spaces of these components: $X_{\infty} = \bigcup _{i =1} ^{\ell} X^{(i)}_{\infty}$.
\end{rmk}

\begin{rmk}\label{rmk:star}
In \cite{NS22}, the authors introduce the condition $(\star)_n$ below, and they study the arc spaces of such $k[t]$-schemes $X$.   
\begin{quote}
$(\star)_n$: \ $X$ is a scheme of finite type over $\operatorname{Spec} k[t]$. 
Any irreducible component of $X$ has dimension at least $n+1$. 
Furthermore, any irreducible component dominating $\operatorname{Spec} k[t]$ is exactly $(n+1)$-dimensional.
\end{quote}
\end{rmk}

\begin{defi}\label{defi:thin}
Let $n$ be a non-negative integer, and 
let $X$ be a scheme of finite type over $k[t]$ satisfying $(\star)_n$. 
A subset $A \subset X_{\infty}$ is called \textit{thin} 
if $A \subset Z_{\infty}$ holds for some closed subscheme $Z$ of $X$ with $\dim Z \le n$. 
\end{defi}

\begin{rmk}\label{rmk:thin_cylinder_kt}
\begin{enumerate}
\item
Suppose that $X$ is a variety over $k[t]$ and $A$ is a cylinder in Definition \ref{defi:thin}. 
Then $A$ is thin if and only if $A \subset W_{\infty}$, 
where $W \subset X$ is the locus defined by the relative Jacobian ideal $\operatorname{Jac}_{X/k[t]}$ (see \cite{NS22}*{Lemma 2.23}). 

\item
By Remark \ref{rmk:bc}, for a variety $X$ over $k$, its arc space $X_{\infty}$ is naturally identified with the arc space of $X \times_{\operatorname{Spec} k} \operatorname{Spec} k[t]$ viewed as a $k[t]$-scheme. 
Under this identification, a cylinder $C \subset X_{\infty}$ is thin in the sense of Definition \ref{defi:thin_k} if and only if the corresponding cylinder is thin in the sense of Definition \ref{defi:thin}. 
This equivalence follows from Remark \ref{rmk:thin_cylinder_k} and the property explained in (1) above. 

\item
If we consider a general subset $C$ of $X_{\infty}$ that is not necessarily a cylinder, the equivalence in (2) no longer holds. 
While any subset that is thin in the sense of Definition \ref{defi:thin_k} is always thin in the sense of Definition \ref{defi:thin}, the converse is not true in general.
\end{enumerate}
\end{rmk}

\subsection{Virtually free actions}
In this subsection, following \cite{NS4}, we define the concept of virtually free actions and summarize their main properties established therein. 

\begin{defi}\label{defi:vf}
Suppose that a finite group $G$ acts on a normal variety $X$. 
We say that the $G$-action on $X$ is \textit{virtually free} if the following condition holds: 
for every $G$-equivariant proper birational morphism $Y \to X$ from a normal variety $Y$, the $G$-action on $Y$ is free in codimension one. 
\end{defi}

This definition admits the following equivalent characterizations.

\begin{prop}[\cite{NS4}*{Proposition 3.8}]\label{prop:vf_equiv}
Suppose that a finite group $G$ acts on a normal variety $X$. 
Then the following conditions are equivalent:
\begin{enumerate}
\item 
The $G$-action on $X$ is virtually free. 
\item
There exists a $G$-equivariant proper birational morphism $Y \to X$ from a normal variety $Y$ such that the $G$-action on $Y$ is free. 
\item 
For every $G$-equivariant proper birational morphism $Y \to X$ from a smooth variety $Y$, the $G$-action on $Y$ is free. 
\end{enumerate}
\end{prop}

\begin{rmk}\label{rmk:vf}
By definition, we have the following implications for a $G$-action on $X$:
\[
\text{free} \implies \text{virtually free} \implies \text{free in codimension one}.
\]
Theorem \ref{thm:kltcase} below shows that the converse ``virtually free $\implies$ free'' holds when $X$ is klt. 
In Example \ref{eg:cubic}, we provide an example of an action that is virtually free but not free.
\end{rmk}

\begin{eg}\label{eg:cubic}
Let $\xi \in k$ be a primitive cube root of unity. 
Let $H \subset \mathbb{A}_k^3$ be the hypersurface defined by $x_1^3 + x_2^3 + x_3^3 = 0$. 
Let $\gamma := \operatorname{diag}(1, \xi, \xi^2) \in \operatorname{GL}_3(k)$. 
Then, the $\langle \gamma \rangle$-action on $H$ is virtually free. 
Indeed, one can see that the induced $\langle \gamma \rangle$-action on the blow-up $\operatorname{Bl}_0 H$ of $H$ at the origin $(0,0,0)$ is free (see \cite{NS4}*{Lemma 4.1} for more details). 
\end{eg}

An important property of virtual freeness is that, as noted in Remark \ref{rmk:vf}, the converse implication 
``$\text{virtually free} \implies \text{free}$'' holds true provided that $X$ is klt. 

\begin{thm}[\cite{NS4}*{Theorem 3.10}]\label{thm:kltcase}
Suppose that a finite group $G$ acts on a klt variety $X$.
If the $G$-action on $X$ is virtually free, then this action is free. 
\end{thm}

The following proposition is also an important property of virtually free actions: it asserts that the minimal log discrepancy is preserved under the quotient, precisely as if the action were free.

\begin{prop}[\cite{NS4}*{Proposition 3.5}]\label{prop:eq}
Suppose that a finite group $G$ acts on a normal $\mathbb{Q}$-Gorenstein variety $X$. 
Suppose that the $G$-action on $X$ is virtually free. 
Let $\mathfrak{a}$ be a non-zero $\mathbb{R}$-ideal sheaf on $X/G$. 
Then, we have 
\[
\operatorname{mld}_{x}(X, \mathfrak{a}\mathcal{O}_X) = \operatorname{mld}_{x'}(X/G, \mathfrak{a})
\]
for any closed point $x \in X$ and its image $x'$ on the quotient variety $X/G$. 
\end{prop}

\section{Virtually free actions and their characterizations}\label{section:vfc}

In this section, we fix a variety $X$ and a finite group $G$ of order $d$. 
Suppose that $G$ acts on $X$. 
We also fix a primitive $d$-th root of unity $\xi \in k$. 
For $d' \in \mathbb{Z}_{>0}$, we define 
\[
T := \operatorname{Spec}k[t], \qquad 
T_{1/d'} := \operatorname{Spec}k[t^{1/d'}]. 
\]

\begin{defi}\label{defi:twisted_quotient}
\begin{enumerate}
\item 
For an element $\gamma \in G$ of order $d'$, we define 
\[
X^{[\gamma]} := (X \times T_{1/d'}) / \langle \gamma \rangle, 
\]
where $\langle \gamma \rangle$ acts diagonally on $X \times T_{1/d'}$ such that the action of $\gamma$ on the coordinate ring of $T_{1/d'}$ is given by $t^{1/d'} \mapsto \xi^{d/d'} t^{1/d'}$. 

\item
Since the second projection $X \times T_{1/d'} \to T_{1/d'}$ is $\gamma$-equivariant and we have a natural identification $T_{1/d'} / \langle \gamma \rangle \cong T$, 
it induces a morphism 
\[
X^{[\gamma]} \to T. 
\]
Therefore, $X^{[\gamma]}$ naturally has the structure of a $k[t]$-scheme. 
We define $X^{[\gamma]}_{\infty}$ as the arc space of the $k[t]$-scheme $X^{[\gamma]}$ in the sense of Subsection \ref{subsection:arc2}. 
\end{enumerate}
\end{defi}

\begin{rmk}\label{rmk:twist}
\begin{enumerate}
\item 
The fiber of $X^{[\gamma]} \to T$ over any non-zero closed point $t = c \neq 0$ is isomorphic to $X$. 
The central fiber over $t = 0$ is generally a non-reduced scheme, and its reduced structure is isomorphic to the quotient variety $X / \langle \gamma \rangle$. 

\item
In particular, $X^{[\gamma]}$ satisfies the condition $(\star)_n$ in Remark \ref{rmk:star} with $n = \dim X$.

\item
Over $T \setminus \{0\}$, the base change $T_{1/d'} \setminus \{0\} \to T \setminus \{0\}$ is a Galois cover with Galois group $\langle \gamma \rangle$. 
Hence, the restriction of $X^{[\gamma]}$ over $T \setminus \{0\}$ coincides with the classical twist of $X$ by the $\langle \gamma \rangle$-torsor $T_{1/d'} \setminus \{0\} \to T \setminus \{0\}$. 
The scheme $X^{[\gamma]}$ can be viewed as a ``branched version'' of the twisted variety associated with the $G$-action on $X$. 
\end{enumerate}
\end{rmk}

We now discuss the basic properties of $X^{[\gamma]}$.

\begin{prop}\label{prop:bp}
For an element $\gamma \in G$ of order $d'$, the following assertions hold. 
\begin{enumerate}
\item 
There exists a natural finite surjective morphism 
\[
X \times T_{1/d'} \to X^{[\gamma]} \times_T T_{1/d'}, 
\]
which induces an isomorphism 
\[
X \times (T_{1/d'} \setminus \{ 0 \}) \xrightarrow{\sim} 
\bigl( X^{[\gamma]} |_{t \neq 0} \bigr) \times_{T \setminus \{ 0 \}} (T_{1/d'} \setminus \{ 0 \})
\]
over $T_{1/d'} \setminus \{ 0 \}$. 

\item 
There exists a one-to-one correspondence between 
the set of sections of $X^{[\gamma]} \to T$ and the set of $\gamma$-equivariant morphisms $T_{1/d'} \to X$. 

\item
Let $X^{[\gamma]}_{\infty}$ be the arc space of the $k[t]$-scheme $X^{[\gamma]}$. 
For any field extension $K$ of $k$, there is a natural bijection
\[
X^{[\gamma]}_{\infty}(K) \simeq \left\{ \operatorname{Spec} K[[t^{1/d'}]] \to X \;\middle|\; \text{\rm $\gamma$-equivariant} \right\}. 
\]
\end{enumerate}
\end{prop}

\begin{proof}
(1) The following two compositions coincide:
\[
X \times T_{1/d'} \to X^{[\gamma]} = (X \times T_{1/d'}) / \langle \gamma \rangle \to T \quad \text{and} \quad 
X \times T_{1/d'} \xrightarrow{p_2} T_{1/d'} \to T. 
\]
Therefore, by the universal property of fiber products, we obtain a morphism 
\[
X \times T_{1/d'} \to X^{[\gamma]} \times_T T_{1/d'}, \quad (x,s) \mapsto \bigl( \overline{(x,s)}, s \bigr).
\]
Since $X^{[\gamma]}$ is the quotient by a finite group, this induced morphism is finite and surjective. The fact that it is an isomorphism outside $t=0$ follows because the $\langle \gamma \rangle$-action on $T_{1/d'}$ is free away from the origin. 

(2) For a $\gamma$-equivariant morphism $f \colon T_{1/d'} \to X$, the map 
\[
T_{1/d'} \to X \times T_{1/d'}, \quad s \mapsto (f(s), s)
\]
is $\gamma$-equivariant, and hence it descends to a morphism $T \to X^{[\gamma]}$, which is a section of $X^{[\gamma]} \to T$. 

Conversely, given a section $\sigma \colon T \to X^{[\gamma]}$ of $X^{[\gamma]} \to T$, the base change induces a section
\[
T_{1/d'} \cong T \times_T T_{1/d'} \xrightarrow{\sigma'} X^{[\gamma]} \times_T T_{1/d'}
\]
of the projection $X^{[\gamma]} \times_T T_{1/d'} \to T_{1/d'}$. 
Since the morphism $q \colon X \times T_{1/d'} \to X^{[\gamma]} \times_T T_{1/d'}$ is finite (in particular, proper) and an isomorphism over $T_{1/d'} \setminus \{ 0 \}$, the valuative criterion of properness guarantees that this section lifts uniquely to a section
\[
\widetilde{f} \colon T_{1/d'} \to X \times T_{1/d'}
\]
of the projection $X \times T_{1/d'} \to T_{1/d'}$. 
\[
\xymatrix{
& X \times T_{1/d'} \ar[d]^{q} \\
T_{1/d'} \ar[ur]^{\widetilde{f}} \ar[r]_-{\sigma'} \ar[d] & X^{[\gamma]} \times_T T_{1/d'} \ar[d] \\
T \ar[r]^-{\sigma} & X^{[\gamma]}
}
\]
This $\widetilde{f}$ is necessarily a $\gamma$-equivariant map. 
Indeed, since $\sigma'$ and $q$ are $\gamma$-equivariant, the conjugated morphism $\gamma^{-1} \circ \widetilde{f} \circ \gamma$ is also a lift of the section $\sigma'$, which must coincide with $\widetilde{f}$ by uniqueness. 
Therefore, $\widetilde{f}$ yields a $\gamma$-equivariant morphism $T_{1/d'} \to X$.

It is straightforward to check that these two correspondences are mutually inverse.

(3) The assertion follows from the exact same argument as in the proof of (2), simply by replacing $T$ and $T_{1/d'}$ with $\operatorname{Spec} K[[t]]$ and $\operatorname{Spec} K[[t^{1/d'}]]$, respectively. 
\end{proof}

\begin{cor}\label{cor:empty}
For an element $\gamma \in G$ of order $d'$, 
the following conditions are equivalent:
\begin{enumerate}
\item
The morphism $X^{[\gamma]} \to T$ has a section. 

\item
There exists a $\gamma$-equivariant morphism $T_{1/d'} \to X$. 

\item
$X^{[\gamma]} _{\infty} \neq \emptyset$. 

\item
$X$ has a $\gamma$-fixed closed point. 
\end{enumerate}
\end{cor}

\begin{proof}
By Proposition \ref{prop:bp}(2), (1) and (2) are equivalent. 

Next, we prove the equivalence (2) $\Leftrightarrow$ (4). 
If $f \colon T_{1/d'} \to X$ is a $\gamma$-equivariant morphism, the point $f(0)$ is necessarily a $\gamma$-fixed point. 
Conversely, for a $\gamma$-fixed point $x \in X$, the constant morphism $f \colon T_{1/d'} \to X$ factoring through $x$ is naturally a $\gamma$-equivariant morphism.

By Proposition \ref{prop:bp}(3), the equivalence (3) $\Leftrightarrow$ (4) is proved in the same way as (2) $\Leftrightarrow$ (4) by replacing $T_{1/d'}$ with $\operatorname{Spec} k[[t^{1/d'}]]$. 
\end{proof}

\begin{prop}\label{prop:vf_thin}
Suppose that a finite group $G$ acts on a normal variety $X$. 
Then, the following conditions are equivalent. 
\begin{enumerate}
\item 
The $G$-action on $X$ is virtually free. 

\item
For any $\gamma \in G \setminus \{ 1_G \}$, $X^{[\gamma]}_{\infty}$ is a thin set of $X^{[\gamma]}_{\infty}$ itself. 
\end{enumerate}
\end{prop}

\begin{proof}
We shall prove the implication (1) $\Rightarrow$ (2). 
Suppose that there exists $\gamma \in G \setminus \{ 1_G \}$ such that $X^{[\gamma]}_{\infty}$ is not a thin set of $X^{[\gamma]}_{\infty}$ itself. 
Let $f \colon Y \to X$ be an arbitrary $G$-equivariant resolution, 
and let $Z \subset X$ be a $G$-invariant closed subset such that $f$ is an isomorphism outside $Z$. 
Then, we can identify $Z^{[\gamma]} := (Z \times T_{1/d'}) / \langle \gamma \rangle$ with a closed subset of $X^{[\gamma]}$, 
and hence, $Z^{[\gamma]}_{\infty}$ with a closed subset of $X^{[\gamma]}_{\infty}$. 
Since $Z^{[\gamma]}_{\infty}$ is a thin set by definition, but $X^{[\gamma]}_{\infty}$ is not a thin set of itself, there exists an arc $\alpha \in X^{[\gamma]}_{\infty} \setminus Z^{[\gamma]}_{\infty}$. 
Since the induced proper morphism $Y^{[\gamma]} \to X^{[\gamma]}$ is an isomorphism over $X^{[\gamma]} \setminus Z^{[\gamma]}$, the valuative criterion of properness guarantees that the arc $\alpha$ lifts to an arc in $Y^{[\gamma]}_{\infty}$. 
Consequently, it follows that $Y^{[\gamma]}_{\infty} \neq \emptyset$. 
By Corollary \ref{cor:empty}, $Y$ has a $\gamma$-fixed closed point. 
Since $f \colon Y \to X$ was an arbitrary resolution, the $G$-action on $X$ is not virtually free by the equivalent condition for a virtually free action in Proposition \ref{prop:vf_equiv}. 

We shall prove the implication (2) $\Rightarrow$ (1). 
Suppose that the $G$-action on $X$ is not virtually free. 
Let $f \colon Y \to X$ be a $G$-equivariant resolution. 
By the equivalent condition for a virtually free action in Proposition \ref{prop:vf_equiv}, 
there exist $\gamma \in G \setminus \{ 1_G \}$ and a closed point $y \in Y$ such that $y$ is a $\gamma$-fixed point. 
Thus, by Corollary \ref{cor:empty}, $Y^{[\gamma]}_{\infty} \neq \emptyset$. 
Finally, by applying Lemma \ref{lem:2.27'} below to the $k[t]$-morphism $Y^{[\gamma]} \to X^{[\gamma]}$, we conclude that $X^{[\gamma]}_{\infty}$ is not a thin set of $X^{[\gamma]}_{\infty}$ itself. 
\end{proof}

\begin{lem}[cf.\ \cite{NS22}*{Lemma 2.27}]\label{lem:2.27'}
Let $f \colon W \to X$ be a proper birational $k[t]$-morphism of $k[t]$-varieties $W$ and $X$. 
Suppose that $W |_{t \not = 0}$ is smooth over $k$. 
If $C \subset X_{\infty}$ is a cylinder satisfying $f_{\infty}^{-1} (C) \not = \emptyset$, 
then $C$ is not a thin set of $X_{\infty}$. 
\end{lem}
\begin{proof}
This can be proved in the same way as \cite{NS22}*{Lemma 2.27}. 
Although \cite{NS22}*{Lemma 2.27} assumes that the whole space $W$ is smooth, its proof only uses the fact that $W$ is smooth over $\operatorname{Spec} k[t]$ outside finitely many points. 
Therefore, the exact same argument applies here. 
\end{proof}

\section{Arc spaces of quotient varieties: general actions}\label{section:dl1}

In this section, we discuss the relationship between the arc space $X'_{\infty}$ of the quotient variety $X' = X/G$ and 
the arc space $X^{[\gamma]}_{\infty}$ introduced in Section \ref{section:vfc}.

\begin{defi}\label{defi:dl1}
Let $X$ be a variety over $k$, and let $G$ be a finite group of order $d$ acting on $X$. 
Let $X' = X /G$ be the quotient variety. 
Let $\xi$ be a primitive $d$-th root of unity.
\begin{enumerate}
\item 
For $\gamma \in G$, we define a subset ${X'}^{(\gamma)} _{\infty} \subset X'_{\infty}$ whose set of $K$-arcs, for any field extension $K$ of $k$, is given by 
\[
{X'}^{(\gamma)} _{\infty}(K) := \left\{ \varphi \in X'_{\infty}(K) \;\middle|\; 
\begin{aligned} 
&\text{$\varphi$ lifts to a $K$-arc $\varphi ' \in X_{\infty} ^{1/d}(K)$} \\ 
&\text{satisfying $\varphi' (\xi t^{1/d}) = \gamma \varphi' (t^{1/d})$} 
\end{aligned} 
\right\}, 
\]
where $X_{\infty} ^{1/d}(K) := \operatorname{Hom}_k \bigl( \operatorname{Spec} K[[t^{1/d}]], X \bigr)$.

\item
Let $\gamma \in G$ be an element of order $d'$. 
Then, the $\gamma$-equivariant morphism
\[
X \times T_{1/d'} \to X' \times T_{1/d'} ; \quad (x,s) \mapsto (\overline{x}, s)
\]
induces a $k[t]$-morphism
\[
h \colon X^{[\gamma]} \to X' \times T, 
\]
which in turn induces a morphism on the arc spaces
\[
h_{\infty} \colon X^{[\gamma]}_{\infty} \to X' _{\infty}
\]
(cf.\ Remark \ref{rmk:bc}).
\end{enumerate}
\end{defi}

\begin{rmk}\label{rmk:dl}
The subset ${X'}^{(\gamma)} _{\infty}$ was introduced by Denef and Loeser in \cite{DL02} to study $X'_{\infty}$ via $X^{(\gamma)}_{\infty}$, which is introduced in Section \ref{section:dl2}. 
By \cite{DL02}*{Section 2.1} (cf.\ \cite{Yas16}*{Section 3}, \cite{NS22}*{Proposition 3.4}), we have
\[
X' _{\infty} \setminus Z' _{\infty} = \bigsqcup _{[\gamma] \in \operatorname{Conj}(G)} 
\bigl( {X'}^{(\gamma)} _{\infty} \setminus Z' _{\infty} \bigr), 
\]
where $Z' \subset X'$ is the minimal closed subset such that the quotient map $X \to X'$ is \'{e}tale outside $Z'$. 
\end{rmk}

\begin{lem}\label{lem:image_arc}
In the setting of Definition \ref{defi:dl1}, the image of $h_{\infty} \colon X^{[\gamma]}_{\infty} \to X' _{\infty}$ coincides with ${X'}^{(\gamma)} _{\infty}$. 
\end{lem}

\begin{proof}
Let $K \supseteq k$ be a field extension. 
Let $\varphi \in X'_{\infty}(K)$ be a $K$-arc. 
By Proposition \ref{prop:bp}(3), $\varphi \in \operatorname{Im} (h_{\infty})$ if and only if there exists a $\gamma$-equivariant morphism $\widetilde{\varphi} \colon \operatorname{Spec} K[[t^{1/d'}]] \to X$ making the following diagram commute: 
\[
\xymatrix{
  X \ar[r]^-{q} & X' \\
  \operatorname{Spec} K[[t^{1/d'}]] \ar[u]^-{\exists \widetilde{\varphi}} \ar[r] & \operatorname{Spec} K[[t]] \ar[u]_-{\varphi}
} 
\]
On the other hand, by definition, $\varphi \in {X'}^{(\gamma)} _{\infty}$ if and only if there exists a morphism $\varphi' \colon \operatorname{Spec} K[[t^{1/d}]] \to X$ satisfying $\varphi'(\xi t^{1/d}) = \gamma \varphi'(t^{1/d})$ that makes the following diagram commute: 
\[
\xymatrix{
  X \ar[r]^-{q} & X' \\
  \operatorname{Spec} K[[t^{1/d}]] \ar[u]^-{\exists \varphi'} \ar[r] & \operatorname{Spec} K[[t]] \ar[u]_-{\varphi}
}
\]
Thus, to complete the proof, it suffices to show that the existence of such a $\widetilde{\varphi}$ is equivalent to the existence of such a $\varphi'$.

First, suppose we are given such a $\widetilde{\varphi}$. 
By composing it with the natural morphism $\operatorname{Spec} K[[t^{1/d}]] \to \operatorname{Spec} K[[t^{1/d'}]]$, we obtain a morphism $\varphi' \colon \operatorname{Spec} K[[t^{1/d}]] \to X$ making the diagram commute. 
Since $\widetilde{\varphi}$ is $\gamma$-equivariant, we have $\widetilde{\varphi}(\xi^{d/d'} t^{1/d'}) = \gamma \widetilde{\varphi}(t^{1/d'})$. 
Therefore, the composed morphism $\varphi'$ satisfies 
\[
\varphi'(\xi t^{1/d}) = \widetilde{\varphi}(\xi^{d/d'} t^{1/d'}) = \gamma \widetilde{\varphi}(t^{1/d'}) = \gamma \varphi'(t^{1/d}), 
\]
which means $\varphi'$ satisfies the required condition. 

Conversely, suppose we are given such a $\varphi'$. 
Since the order of $\gamma$ is $d'$, we have  
\[
\varphi'(\xi^{d'} t^{1/d}) = \gamma^{d'} \varphi'(t^{1/d}) = \varphi'(t^{1/d}). 
\]
This equality means that the morphism $\varphi'$ is invariant under the $\xi^{d'}$-action on $\operatorname{Spec} K[[t^{1/d}]]$. 
Therefore, $\varphi'$ uniquely descends to a morphism 
\[
\widetilde{\varphi} \colon \operatorname{Spec} K[[t^{1/d'}]] \to X. 
\]
Because $\varphi'$ is the composition of $\operatorname{Spec} K[[t^{1/d}]] \to \operatorname{Spec} K[[t^{1/d'}]]$ and $\widetilde{\varphi}$, the initial relation $\varphi'(\xi t^{1/d}) = \gamma \varphi'(t^{1/d})$ yields the relation 
\[
\widetilde{\varphi}(\xi^{d/d'} t^{1/d'}) = \gamma \widetilde{\varphi}(t^{1/d'}), 
\]
ensuring that $\widetilde{\varphi}$ is $\gamma$-equivariant.
\end{proof}

\begin{rmk}\label{rmk:identity_thin}
In the following, we consider the thinness of subsets of $X'_{\infty}$. 
Here, we identify $X'_{\infty}$ with the arc space $(X' \times T)_{\infty}$ of $X' \times T$ viewed as a $k[t]$-scheme (see Remark \ref{rmk:bc}), and thinness is understood in the sense of Definition \ref{defi:thin}. 
As noted in Remark \ref{rmk:thin_cylinder_kt}(3), it should be stressed that this notion of thinness is weaker than the usual one in the sense of Definition \ref{defi:thin_k}.
\end{rmk}

\begin{prop}\label{prop:thin_eq}
In the setting of Definition \ref{defi:dl1}, let $\gamma \in G$. 
For a subset $C \subset X'_{\infty}$, the following conditions are equivalent:
\begin{enumerate}
\item 
$C \cap {X'}^{(\gamma)}_{\infty}$ is a thin set of $X'_{\infty}$ (see Remark \ref{rmk:identity_thin}). 

\item 
$h_{\infty}^{-1} (C)$ is a thin set of $X^{[\gamma]}_{\infty}$. 
\end{enumerate}
\end{prop}

\begin{proof}
First, we prove (1) $\Rightarrow$ (2). 
Suppose that $C \cap {X'}^{(\gamma)}_{\infty}$ is a thin set of $X'_{\infty}$. 
By definition (see Remark \ref{rmk:identity_thin}), there exists a proper closed subset $W' \subsetneq X' \times T$ such that $C \cap {X'}^{(\gamma)}_{\infty} \subset W' _{\infty}$. 
By Lemma \ref{lem:image_arc}, we have 
\[
h_{\infty}^{-1}(C) = h_{\infty}^{-1}(C \cap {X'}^{(\gamma)}_{\infty}) \subset h_{\infty} ^{-1} (W'_{\infty}) = (h^{-1}(W'))_{\infty}. 
\]
Since $W' \subsetneq X' \times T$ is a proper closed subset, $h^{-1}(W')$ is a proper closed subset of $X^{[\gamma]}$. 
Thus, $h_{\infty}^{-1}(C)$ is a thin set of $X^{[\gamma]}_{\infty}$. 

Next, we prove (2) $\Rightarrow$ (1). 
Suppose that $h_{\infty}^{-1} (C)$ is a thin set of $X^{[\gamma]}_{\infty}$. 
By definition, there exists a proper closed subset $W \subsetneq X^{[\gamma]}$ such that $h_{\infty}^{-1}(C) \subset W_{\infty}$. 
Applying $h_{\infty}$ to both sides, and using Lemma \ref{lem:image_arc} again, we obtain 
\[
C \cap {X'}^{(\gamma)}_{\infty} = h_{\infty}(h_{\infty}^{-1}(C)) \subset h_{\infty} (W_{\infty}) \subset W' _{\infty} 
\]
for $W' := h(W)$. 
Since $h$ is a finite morphism by construction, 
$W'$ is a proper closed subset of $X'  \times T$, which proves that $C \cap {X'}^{(\gamma)}_{\infty}$ is a thin set of $X'_{\infty}$ (see Remark \ref{rmk:identity_thin}). 
\end{proof}

\begin{cor}\label{cor:vf_thin}
Suppose that a finite group $G$ acts on a normal variety $X$. 
Let $X' = X/G$ be the quotient variety. 
Then, the following conditions are equivalent. 
\begin{enumerate}
\item 
The $G$-action on $X$ is virtually free. 

\item
For any $\gamma \in G \setminus \{ 1_G \}$, $X^{[\gamma]}_{\infty}$ is a thin set of $X^{[\gamma]}_{\infty}$ itself. 

\item
For any $\gamma \in G \setminus \{ 1_G \}$, ${X'}^{(\gamma)}_{\infty}$ is a thin set of $X'_{\infty}$ (see Remark \ref{rmk:identity_thin}). 
\end{enumerate}
\end{cor}

\begin{proof}
The assertion follows from Propositions \ref{prop:vf_thin} and \ref{prop:thin_eq}. 
\end{proof}

\section{Arc spaces of quotient varieties: linear actions}\label{section:dl2}
In this section, we focus on the case where a finite group $G$ acts linearly on an affine space. 
In this setting, Denef and Loeser \cite{DL02} introduced a certain $k[t]$-scheme, denoted by $X^{(\gamma)}$, to investigate the arc space of the quotient variety (cf.\ \cite{NS22}). 
The goal of this section is to recall the construction of $X^{(\gamma)}$ and to compare it with the scheme $X^{[\gamma]}$ introduced in Section \ref{section:vfc}. 
In particular, we will establish a natural bijection between their sets of $K$-arcs for any field extension $K$ (Proposition \ref{prop:()to[]}).

Let $N$ be a positive integer, and let $R = k[x_1, \ldots, x_N]$. 
Let $G \subset \mathrm{GL}_N (k)$ be a finite subgroup of order $d$. 
Suppose that $G$ linearly acts on $\mathbb{A}^N = \operatorname{Spec} R$. 
Let $\xi$ be a primitive $d$-th root of unity.

Let $\gamma \in G$ be an element of order $d'$. 
Since $G$ is a finite group, $\gamma$ can be diagonalized with respect to a new basis $x_1^{(\gamma)}, \ldots, x_N^{(\gamma)}$ consisting of $k$-linear combinations of the original basis $x_1, \ldots , x_N$.
We may assume that the action of $\gamma$ is represented by the diagonal matrix $\operatorname{diag} \left( \xi ^{e_1}, \ldots , \xi ^{e_N} \right)$ with respect to this basis, where $0 \le e_i \le d-1$. 
We define a $k$-algebra homomorphism ${\lambda}^* _{\gamma} \colon R \to R[t^{1/d}]$ as follows:
\[
{\lambda}^* _{\gamma} \colon R \to R[t^{1/d}]; 
\quad x_i^{(\gamma)} \mapsto t^{e_i/d}x_i^{(\gamma)}. 
\]
We note that the ring homomorphism ${\lambda}^* _{\gamma}$ does not depend on the choice of the basis $x_1^{(\gamma)}, \ldots, x_N^{(\gamma)}$ as long as $\gamma$ is diagonalized with respect to it. 

By trivially extending ${\lambda}^* _{\gamma}$ to $R[t^{1/d'}]$ so that ${\lambda}^* _{\gamma} \bigl( t^{1/d'} \bigr) = t^{1/d'}$, it induces a $k[t]$-algebra homomorphism
\[
{\lambda}^* _{\gamma} \colon \bigl( R[t^{1/d'}] \bigr)^{\langle \gamma \rangle} \to R[t], 
\]
where the $\gamma$-action on $t^{1/d'}$ is defined by $\gamma \bigl( t^{1/d'} \bigr) = \xi ^{d/d'} t^{1/d'}$. 
We note that the image of this homomorphism indeed falls into $R[t]$. 
In fact, an arbitrary monomial $\bigl( t^{1/d'} \bigr)^m \prod \bigl( x_j^{(\gamma)} \bigr)^{m_j}$ belongs to the invariant subring $\bigl( R[t^{1/d'}] \bigr)^{\langle \gamma \rangle}$ if and only if $\xi^{m(d/d') + \sum m_j e_j} = 1$, which means $m(d/d') + \sum m_j e_j$ is a multiple of $d$. 
Since its image under $\lambda^*_{\gamma}$ is given by $t^{m/d' + \sum m_j e_j / d} \prod \bigl( x_j^{(\gamma)} \bigr)^{m_j}$, the exponent of $t$ is exactly $\frac{1}{d} \bigl( m(d/d') + \sum m_j e_j \bigr)$, which is an integer. 
Geometrically, ${\lambda}^* _{\gamma}$ induces a morphism of schemes over $T$:
\[
\lambda_{\gamma} \colon \mathbb{A}^N \times T \to ( \mathbb{A}^N )^{[\gamma]}. 
\]

Let $f_1, \ldots , f_{c} \in R^G$ be a regular sequence which is contained in the maximal ideal $(x_1, \ldots , x_N) ^G$. 
We define 
\[
X := \operatorname{Spec} \bigl( R/(f_1, \ldots, f_{c}) \bigr) \subset \mathbb{A}^N, \quad
X' := \operatorname{Spec} \bigl( R^G/(f_1, \ldots, f_{c}) \bigr). 
\]
Then, we have $X' \cong X/G$. 
Suppose that $X$ is normal. 

Since the polynomials $f_i$ are $G$-invariant (in particular, $\gamma$-invariant), we have 
\[
F_i := \lambda^* _{\gamma} (f_i) \in R[t]. 
\]
We define a $k[t]$-scheme
\[
X^{(\gamma)} := \operatorname{Spec} \bigl( R[t]/(F_1, \ldots, F_c) \bigr) \subset \mathbb{A}^N \times T. 
\]
Then, $\lambda_{\gamma} \colon \mathbb{A}^N \times T \to ( \mathbb{A}^N )^{[\gamma]}$ induces a morphism 
\[
\lambda_{\gamma} \colon X^{(\gamma)} \to X^{[\gamma]}. 
\]
Let $X^{(\gamma)} _{\infty}$ denote its arc space as a $k[t]$-scheme. 

\begin{rmk}\label{rmk:iso_out0}
The morphism $\lambda_{\gamma} \colon X^{(\gamma)} \to X^{[\gamma]}$ is an isomorphism over $T \setminus \{0\}$. 
Indeed, over the open subset where $t \neq 0$, the ring homomorphism $\lambda_{\gamma}^* \colon \bigl( R[t^{1/d'}] \bigr)^{\langle \gamma \rangle} \to R[t]$ has an inverse defined by 
\[
x_i^{(\gamma)} \mapsto t^{-e_i/d} x_i^{(\gamma)}. 
\]
Note here that the element $t^{-e_i/d} x_i^{(\gamma)}$ is $\gamma$-invariant. 
Therefore, $\lambda_{\gamma}$ gives an isomorphism outside $t=0$. 
\end{rmk}

\begin{rmk}
In particular, $X^{(\gamma)}$ satisfies the condition $(\star)_n$ in Remark \ref{rmk:star} with $n := \dim X = N-c$.
\end{rmk}

\begin{prop}\label{prop:()to[]}
The morphism $\lambda_{\gamma} \colon X^{(\gamma)} \to X^{[\gamma]}$ induces a bijection $\lambda_{\gamma, \infty} \colon X^{(\gamma)}_{\infty}(K) \xrightarrow{\sim} X^{[\gamma]}_{\infty}(K)$ for any field extension $K$ of $k$. 
Furthermore, this bijection is a homeomorphism with respect to the induced Zariski topology.
\end{prop}

\begin{proof}
Since the inverse image of $X^{[\gamma]}$ by the morphism $\lambda_{\gamma} \colon \mathbb{A}^N \times T \to (\mathbb{A}^N)^{[\gamma]}$ coincides with $X^{(\gamma)}$ by definition, it is sufficient to show the assertions for the case $X = \mathbb{A}^N = \operatorname{Spec} R$. 

A $K$-arc in $(\mathbb{A}^N \times T)_{\infty}$ corresponds to a $k[t]$-algebra homomorphism $R[t] \to K[[t]]$, which is uniquely determined by assigning $x_i^{(\gamma)} \mapsto \alpha_i(t) \in K[[t]]$ for each $i$. 
On the other hand, a $K$-arc in $(\mathbb{A}^N)^{[\gamma]}_{\infty}$ corresponds to a $\gamma$-equivariant $k$-algebra homomorphism $\varphi^* \colon R \to K[[t^{1/d'}]]$. 
Since $\gamma$ acts on $x_i^{(\gamma)}$ by $\xi^{e_i}$ and on $t^{1/d'}$ by $\xi^{d/d'}$, this equivariance condition means 
\[
\xi^{e_i} \varphi^* \bigl( x_i^{(\gamma)} \bigr) \bigl( t^{1/d'} \bigr) = \varphi^* \bigl( x_i^{(\gamma)} \bigr) \bigl( \xi^{d/d'} t^{1/d'} \bigr). 
\]
This equality forces all non-zero terms in the power series $\varphi^* \bigl( x_i^{(\gamma)} \bigr) = \sum c_m \bigl( t^{1/d'} \bigr)^m \in K[[t^{1/d'}]]$ to satisfy $\xi^{m(d/d')} = \xi^{e_i}$. 
Since $\xi$ is a primitive $d$-th root of unity, this is equivalent to $m(d/d') \equiv e_i \pmod d$, which means $m/d' \equiv e_i/d \pmod{1}$. 
Therefore, $\varphi^* \bigl( x_i^{(\gamma)} \bigr)$ can be uniquely factored as 
\[
\varphi^* \bigl( x_i^{(\gamma)} \bigr) = t^{e_i/d} \alpha_i(t)
\]
for some power series $\alpha_i(t) \in K[[t]]$. 
Note that since $\gamma^{d'} = 1$, we have $\xi^{e_i d'} = 1$, implying $e_i$ is a multiple of $d/d'$. Thus $t^{e_i/d}$ is an integer power of $t^{1/d'}$. 

Since the morphism $\lambda_{\gamma}$ is defined by the ring homomorphism $\lambda_{\gamma}^* \bigl( x_i^{(\gamma)} \bigr) = t^{e_i/d} x_i^{(\gamma)}$, the induced map on the arc spaces exactly sends the arc $x_i^{(\gamma)} \mapsto \alpha_i(t)$ to the $\gamma$-equivariant arc $\varphi^* \bigl( x_i^{(\gamma)} \bigr) = t^{e_i/d} \alpha_i(t)$. 
This correspondence gives a natural bijection between the sets of $K$-arcs. 

Finally, we observe that this bijection is a homeomorphism. 
As shown above, a $K$-arc in either $X^{(\gamma)}_{\infty}(K)$ or $X^{[\gamma]}_{\infty}(K)$ is determined by the exact same set of coefficients of the power series $\alpha_i(t) \in K[[t]]$ for $1 \le i \le N$. 
Both spaces of $K$-arcs are naturally identified with the inverse limit of affine spaces $\varprojlim_m \mathbb{A}^{N(m+1)}_K$ parameterized by these coefficients. 
Since the bijection $\lambda_{\gamma, \infty}$ simply acts as the identity map under this natural identification of coordinates, it is a homeomorphism. 
\end{proof}

\begin{prop}\label{prop:thin_eq2}
For a subset $C \subset X^{(\gamma)}_{\infty}$, the following conditions are equivalent:
\begin{enumerate}
\item 
$C$ is a thin set of $X^{(\gamma)}_{\infty}$. 

\item 
$\lambda_{\gamma, \infty} (C)$ is a thin set of $X^{[\gamma]}_{\infty}$. 
\end{enumerate}
\end{prop}

\begin{proof}
The assertion follows from Remarks \ref{rmk:dom} and \ref{rmk:iso_out0}. 
Indeed, by Remark \ref{rmk:dom}, the arc spaces $X^{(\gamma)}_{\infty}$ and $X^{[\gamma]}_{\infty}$ only depend on the irreducible components dominating $\operatorname{Spec} k[t]$. 
Since $\lambda_{\gamma}$ is an isomorphism over the open subset $t \neq 0$ by Remark \ref{rmk:iso_out0}, it induces a natural bijection between the proper closed subsets of $X^{(\gamma)}$ dominating $\operatorname{Spec} k[t]$ and those of $X^{[\gamma]}$. 
Therefore, $C$ is contained in the arc space of a proper closed subset if and only if $\lambda_{\gamma, \infty}(C)$ is, which completes the proof. 
\end{proof}

\begin{eg}\label{eg:shiki}
We consider the case $N = 3$, $c = 1$, and $R = k[x_1, x_2, x_3]$. 
Let $f_1 = x_1^3 + x_2^3 + x_3^3$. 
Let $G = \langle \gamma \rangle$ be a cyclic group of order $d = 3$. 
Let $\xi$ be a primitive cube root of unity. 
Suppose that the representation matrix of $\gamma$ with respect to the basis $x_1, x_2, x_3$ is $\operatorname{diag} (1, \xi, \xi^2)$. 
Thus, the weights are $e_1 = 0, e_2 = 1$, and $e_3 = 2$. 

By definition, the homomorphism $\lambda^*_{\gamma} \colon R \to R[t^{1/3}]$ is given by
\[
\lambda^*_{\gamma}(x_1) = x_1, \quad \lambda^*_{\gamma}(x_2) = t^{1/3}x_2, \quad \lambda^*_{\gamma}(x_3) = t^{2/3}x_3.
\]
Then, the invariant polynomial $f_1$ is mapped to 
\[
F_1 := \lambda^*_{\gamma}(f_1) = x_1^3 + t x_2^3 + t^2 x_3^3 \in R[t]. 
\]
Therefore, the $k[t]$-scheme $X^{(\gamma)}$ is given by
\[
X^{(\gamma)} = \operatorname{Spec} \bigl( k[x_1, x_2, x_3, t] / (x_1^3 + t x_2^3 + t^2 x_3^3) \bigr). 
\]
On the other hand, the $\gamma$-invariant subring $(R[t^{1/3}])^{\langle \gamma \rangle}$ is generated by the invariant monomials $x_1, x_2^3, x_3^3, x_2 x_3, t, t^{2/3}x_2$, and $t^{1/3}x_3$. 
Thus, $X^{[\gamma]}$ is given by
\[
X^{[\gamma]} = \operatorname{Spec} \bigl( k[x_1, x_2^3, x_3^3, x_2 x_3, t, t^{2/3}x_2, t^{1/3}x_3] / (x_1^3 + x_2^3 + x_3^3) \bigr). 
\]

We note that the fibers of the morphisms $X^{(\gamma)} \to T$ and $X^{[\gamma]} \to T$ over $t \neq 0$ are isomorphic (see Remark \ref{rmk:iso_out0}). 
However, their special fibers over $t = 0$ are different. 
The special fiber $X^{(\gamma)}_0$ is given by 
\[
X^{(\gamma)}_0 = \operatorname{Spec} \bigl( k[x_1, x_2, x_3] / (x_1^3) \bigr), 
\]
which is a non-reduced scheme. 
On the other hand, the special fiber $X^{[\gamma]}_0$ is given by 
\[
X^{[\gamma]}_0 = \operatorname{Spec} \bigl( k[x_1, x_2^3, x_3^3, x_2 x_3, t, t^{2/3}x_2, t^{1/3}x_3] / (x_1^3 + x_2^3 + x_3^3, t) \bigr).
\]
Its reduced structure is exactly the quotient variety $X'$, which is isomorphic to 
\[
\operatorname{Spec} \bigl( k[x_1, x_2^3, x_3^3, x_2 x_3] / (x_1^3 + x_2^3 + x_3^3) \bigr). 
\]
\end{eg}

\section{Virtually free actions: localized setting}\label{section:vf_local}

We begin by introducing a localized version of a virtually free action for a fixed element $\gamma \in G$ at a closed point $x \in X$. 

\begin{defi}\label{defi:vf_local}
Suppose that a finite group $G$ acts on a normal variety $X$. 
\begin{enumerate}
\item
For a closed point $x \in X$ and an element $\gamma \in G$, 
we say that the $\gamma$-action on $X$ is \textit{virtually free} at $x$ if, for any $\gamma$-equivariant proper birational morphism $f \colon Y \to X$ from a normal variety $Y$, we have
\[
\dim  \bigl( \operatorname{Fix}(\gamma) \cap f^{-1}(x) \bigr) \le \dim X - 2, 
\]
where $\operatorname{Fix}(\gamma)$ is the set of $\gamma$-fixed closed points in $Y$. 

\item
For a closed point $x \in X$, 
we say that the $G$-action on $X$ is \textit{virtually free} at $x$ if 
the $\gamma$-action on $X$ is virtually free at $x$ for all $\gamma \in G \setminus \{ 1_G \}$. 
\end{enumerate}
\end{defi}

Similar to Proposition \ref{prop:vf_equiv}, we have the following equivalent characterizations for the local definition of virtually free actions.

\begin{prop}\label{prop:vf_local_eq}
Suppose that a finite group $G$ acts on a normal variety $X$. 
For a closed point $x \in X$ and an element $\gamma \in G$, the following conditions are equivalent:
\begin{enumerate}
\item 
The $\gamma$-action on $X$ is virtually free at $x$. 

\item
There exists a $\gamma$-equivariant proper birational morphism $f \colon Y \to X$ from a normal variety $Y$ such that $f^{-1}(x)$ has no $\gamma$-fixed closed point. 

\item
For every $\gamma$-equivariant proper birational morphism $f \colon Y \to X$ from a smooth variety $Y$, $f^{-1}(x)$ has no $\gamma$-fixed closed point.
\end{enumerate}
\end{prop}

\begin{proof}
The equivalence follows from the same argument as in \cite{NS4}*{Proposition 3.8}.
\end{proof}

The following proposition shows that the global definition of virtually free actions is compatible with its localized version.

\begin{prop}\label{prop:vf_global_local}
Suppose that a finite group $G$ acts on a normal variety $X$. 
Then, the following conditions are equivalent:
\begin{enumerate}
\item 
The $G$-action on $X$ is virtually free (see Definition \ref{defi:vf}). 

\item
The $G$-action on $X$ is virtually free at every closed point $x \in X$. 
\end{enumerate}
\end{prop}

\begin{proof}
The implication (2) $\Rightarrow$ (1) follows from Propositions \ref{prop:vf_equiv} and \ref{prop:vf_local_eq}. 

We shall prove the implication (1) $\Rightarrow$ (2) by contraposition. 
Suppose that condition (2) does not hold. 
Then, by Definition \ref{defi:vf_local}, there exist an element $\gamma \in G \setminus \{ 1_G \}$, a closed point $x \in X$, and a $\gamma$-equivariant proper birational morphism $f \colon Y \to X$ from a normal variety $Y$ such that $\dim \bigl( \operatorname{Fix}(\gamma) \cap f^{-1}(x) \bigr) \ge \dim X - 1$. 
In particular, there exists a prime divisor $E$ on $Y$ which is mapped to $x$ and on which $\gamma$ acts trivially. 
By \cite{NS4}*{Lemma 3.3}, we can take a $G$-equivariant proper birational morphism $Y' \to X$ from a normal variety $Y'$ such that $E$ appears as a prime divisor on $Y'$. 
Furthermore, by \cite{NS4}*{Lemma 3.6}, the induced $\gamma$-action on $E \subset Y'$ is also trivial. 
Thus, the $G$-action on $Y'$ is not free in codimension one, and therefore, the $G$-action on $X$ is not virtually free. 
\end{proof}

\begin{defi}\label{defi:arc_local}
Suppose that a finite group $G$ acts on a normal variety $X$. 
Let $\gamma \in G$ be an element of order $d'$, and let $x'$ be a closed point in the quotient variety $X' = X/G$. 
\begin{enumerate}
\item
We define the \textit{local arc spaces} by
\[
X'_{\infty, x'} := \operatorname{Cont}^{\ge 1} (\mathfrak{m}_{x'}) \subset X'_{\infty}, \qquad
{X'}^{(\gamma)}_{\infty, x'} := {X'}^{(\gamma)}_{\infty} \cap \operatorname{Cont}^{\ge 1} (\mathfrak{m}_{x'}), 
\]
where $\mathfrak{m}_{x'} \subset \mathcal{O}_{X'}$ is the maximal ideal sheaf defining $x'$, and ${X'}^{(\gamma)}_{\infty}$ is the subset of $X'_{\infty}$ defined in Definition \ref{defi:dl1}. 
We also define 
\[
X^{[\gamma]}_{\infty, x'} := h_{\infty} ^{-1} \bigl( \operatorname{Cont}^{\ge 1} (\mathfrak{m}_{x'}) \bigr) \subset X^{[\gamma]}_{\infty}, 
\]
where $h_{\infty} \colon X^{[\gamma]}_{\infty} \to X'_{\infty}$ is the morphism defined in Definition \ref{defi:dl1}.

\item 
A $k$-arc $\varphi \in X'_{\infty}$ is called the \textit{trivial arc} associated to $x'$ if the corresponding morphism $\varphi \colon \operatorname{Spec} k[[t]] \to X'$ factors through the closed immersion $x' \hookrightarrow X'$. 
By definition, such $\varphi$ satisfies $\varphi \in X'_{\infty, x'}$, and
$\varphi$ is also called the \textit{trivial arc} of $X'_{\infty, x'}$. 

\item 
Let $x \in X$ be a $\gamma$-fixed closed point mapping to $x'$. 
A $k$-arc $\varphi \in X^{[\gamma]}_{\infty}$ is called the \textit{trivial arc} associated to $x$ if the corresponding $\gamma$-equivariant morphism $\widetilde{\varphi} \colon \operatorname{Spec} k[[t^{1/d'}]] \to X$ factors through the closed immersion $x \hookrightarrow X$ (see Proposition \ref{prop:bp}(3) for the correspondence). 
By definition, such $\varphi$ satisfies $\varphi \in X^{[\gamma]}_{\infty, x'}$, and 
$\varphi$ is also called a \textit{trivial arc} of $X^{[\gamma]}_{\infty, x'}$. 
We note that, geometrically, the $\gamma$-invariant closed subscheme $\{x\} \times T_{1/d'} \subset X \times T_{1/d'}$ descends to a section $s_x \colon T \to X^{[\gamma]}$, and this trivial arc is exactly the arc given by the composition $\operatorname{Spec} k[[t]] \to T \xrightarrow{s_x} X^{[\gamma]}$.
\end{enumerate}
\end{defi}

The following lemma provides a local version of Corollary \ref{cor:empty}. 
The existence of a trivial arc in $X^{[\gamma]}_{\infty, x'}$ naturally characterizes whether the corresponding fiber over $x'$ contains a $\gamma$-fixed point. 

\begin{lem}\label{lem:empty_local}
Suppose that a finite group $G$ acts on a normal variety $X$. 
Let $\gamma \in G$ be an element, and let $x'$ be a closed point in the quotient variety $X' = X/G$. 
Then, the following conditions are equivalent: 
\begin{enumerate}
\item
$X^{[\gamma]} _{\infty, x'} \neq \emptyset$. 

\item
There exists a $\gamma$-fixed closed point $x \in X$ mapping to $x'$. 
\end{enumerate}
\end{lem}

\begin{proof}
Suppose $X^{[\gamma]} _{\infty, x'} \neq \emptyset$. 
By definition, there exists an arc $\varphi \in X^{[\gamma]} _{\infty}$ which maps to an arc in $X'_{\infty}$ centered at $x'$. 
By Proposition \ref{prop:bp}(3), $\varphi$ corresponds to a $\gamma$-equivariant morphism $\widetilde{\varphi} \colon \operatorname{Spec} k[[t^{1/d'}]] \to X$. 
Let $x := \widetilde{\varphi}(0) \in X$ be the center of this morphism. 
Since the induced arc on $X'$ is centered at $x'$, the image of $x$ in $X'$ is exactly $x'$. 
Furthermore, by the $\gamma$-equivariance of $\widetilde{\varphi}$, its center $x$ must be a $\gamma$-fixed point. 

Conversely, suppose that there exists a $\gamma$-fixed closed point $x \in X$ mapping to $x'$. 
By Definition \ref{defi:arc_local}(3), the trivial arc associated to $x$ belongs to $X^{[\gamma]}_{\infty, x'}$. 
In particular, we have $X^{[\gamma]} _{\infty, x'} \neq \emptyset$.  
\end{proof}

The following proposition provides a local version of Corollary \ref{cor:vf_thin}.

\begin{thm}\label{thm:vf_thin_local}
Suppose that a finite group $G$ acts on a normal variety $X$. 
Let $\gamma \in G$ be an element, and let $x'$ be a closed point in the quotient variety $X' = X/G$. 
Then, the following conditions are equivalent: 
\begin{enumerate}
\item 
For every closed point $x \in X$ mapping to $x'$, the $\gamma$-action on $X$ is virtually free at $x$. 

\item
${X'}^{(\gamma)}_{\infty, x'}$ is a thin set of $X'_{\infty}$ see Remark \ref{rmk:identity_thin}). 

\item
$X^{[\gamma]}_{\infty, x'}$ is a thin set of $X^{[\gamma]}_{\infty}$. 
\end{enumerate}
\end{thm}

\begin{proof}
The equivalence (2) $\Leftrightarrow$ (3) follows from Proposition \ref{prop:thin_eq}. 

Next, we prove the implication (1) $\Rightarrow$ (3) by contraposition. 
Suppose that $X^{[\gamma]}_{\infty, x'}$ is not a thin set of $X^{[\gamma]}_{\infty}$. 
Let $f \colon Y \to X$ be an arbitrary $\gamma$-equivariant proper birational morphism from a smooth variety $Y$, 
and let $Z \subsetneq X$ be a $\gamma$-invariant proper closed subset such that $f$ is an isomorphism outside $Z$. 
Then, we can identify $Z^{[\gamma]}$ with a closed subset of $X^{[\gamma]}$, 
and hence, $Z^{[\gamma]}_{\infty}$ with a closed subset of $X^{[\gamma]}_{\infty}$. 
Since $Z^{[\gamma]}_{\infty}$ is a thin set by definition, but $X^{[\gamma]}_{\infty, x'}$ is not a thin set of $X^{[\gamma]}_{\infty}$, there exists an arc $\alpha \in X^{[\gamma]}_{\infty, x'} \setminus Z^{[\gamma]}_{\infty}$. 
Since the induced proper $k[t]$-morphism $Y^{[\gamma]} \to X^{[\gamma]}$ is an isomorphism over $X^{[\gamma]} \setminus Z^{[\gamma]}$, the valuative criterion of properness guarantees that the arc $\alpha$ lifts to an arc $\beta \in Y^{[\gamma]}_{\infty}$. 
Let $y \in Y$ be the $\gamma$-fixed point corresponding to the center of $\beta$ evaluated at $t=0$ (see Lemma \ref{lem:empty_local}). 
Let $x := f(y) \in X$ be its image. 
Because $\alpha$ maps to an arc centered at $x'$ in $X'$, its center $x$ must map to $x'$. 
Furthermore, since $f$ is $\gamma$-equivariant, $x$ is also a $\gamma$-fixed point. 
Thus, we have found a closed point $x \in X$ mapping to $x'$ such that $f^{-1}(x)$ contains a $\gamma$-fixed point $y$. 
By Proposition \ref{prop:vf_local_eq}, this implies that the $\gamma$-action on $X$ is not virtually free at $x$.

Next, we prove the implication (3) $\Rightarrow$ (1) by contraposition. 
Suppose that there exists a closed point $x \in X$ mapping to $x'$ such that the $\gamma$-action on $X$ is not virtually free at $x$. 
By Proposition \ref{prop:vf_local_eq}, there exists a $\gamma$-equivariant proper birational morphism $f \colon Y \to X$ from a smooth variety $Y$ such that $f^{-1}(x)$ contains a $\gamma$-fixed closed point $y$. 
Because $y$ is a $\gamma$-fixed point, the trivial arc $\beta \in Y^{[\gamma]}_{\infty}$ associated to $y$ exists. 
Then, by definition, its image $f^{[\gamma]}_{\infty} (\beta)$ is the trivial arc of $X^{[\gamma]}_{\infty, x'}$ associated to $x$, where $f^{[\gamma]} \colon Y^{[\gamma]} \to X^{[\gamma]}$ is the $k[t]$-morphism induced by $f$. 
In particular, we have $\bigl( f^{[\gamma]}_{\infty} \bigr)^{-1} \bigl( X^{[\gamma]}_{\infty, x'} \bigr) \neq \emptyset$. 
Therefore, by Lemma \ref{lem:2.27'}, the cylinder $X^{[\gamma]}_{\infty, x'}$ is not a thin set of $X^{[\gamma]}_{\infty}$.
\end{proof}

The following is a local version of Theorem \ref{thm:kltcase}.

\begin{thm}\label{thm:klt_local}
Suppose that a finite group $G$ acts on a klt variety $X$. 
Let $\gamma \in G$ be an element. 
If $x \in X$ is a $\gamma$-fixed closed point, then the $\gamma$-action on $X$ is not virtually free at $x$. 
\end{thm}

\begin{proof}
Let $f \colon Y \to X$ be a $\gamma$-equivariant resolution, and let $f^{[\gamma]} \colon Y^{[\gamma]} \to X^{[\gamma]}$ be the induced $k[t]$-morphism. 
Since $x \in X$ is a $\gamma$-fixed point, it induces a section $\sigma := s_x \colon T \to X^{[\gamma]}$ as defined in Definition \ref{defi:arc_local}(3). 
Note that $X^{[\gamma]}$ is klt because $X$ is klt. 
Therefore, by \cite{HM07}*{Corollary 1.7(2)}, the section $\sigma$ lifts to a section $\widetilde{\sigma} \colon T \to Y^{[\gamma]}$ of $Y^{[\gamma]} \to T$. 
The point $\widetilde{\sigma}(0)$ then yields a $\gamma$-fixed point in $f^{-1}(x)$. 
By Proposition \ref{prop:vf_local_eq}(3), this proves that the $\gamma$-action is not virtually free at $x$.
\end{proof}

\begin{rmk}
Theorem \ref{thm:klt_local} can also be proved by the same argument as in the proof of \cite{NS4}*{Theorem 3.10}. 
Note that the original proof of \cite{NS4}*{Theorem 3.10} relies on Takayama's theorem \cite{Tak03}, which asserts the simple connectedness of the fibers of a resolution of a klt germ, 
instead of the result of Hacon and McKernan \cite{HM07}.
\end{rmk}

The following proposition generalizes Proposition \ref{prop:eq}. 

\begin{prop}\label{prop:eq_local}
Suppose that a finite group $G$ acts on a normal $\mathbb{Q}$-Gorenstein variety $X$, and that the action is free in codimension one.
Let $\mathfrak{a}$ be a non-zero $\mathbb{R}$-ideal sheaf on $X/G$. 
Let $x \in X$ be a closed point, and let $x' \in X/G$ be its image. 
Let $\operatorname{NVF}(X,x)$ be the set of elements $\gamma \in G$ such that the $\gamma$-action on $X$ is not virtually free at $x$.  
Then, we have 
\[
\operatorname{mld}_{x'}(X/G, \mathfrak{a}) = 
\min _{\gamma \in \operatorname{NVF}(X,x)}
\operatorname{mld}_{x' _{\gamma}} \bigl( X/\langle \gamma \rangle , \mathfrak{a} \mathcal{O}_{X/\langle \gamma \rangle} \bigr), 
\]
where $x'_{\gamma}$ denotes the image of $x$ in $X/\langle \gamma \rangle$. 
\end{prop}

\begin{proof}
This can be proved by following the same argument as in \cite{NS22}*{Remark 4.13}. 
The key point is that the cyclic group $G_F$ appearing in the proof of \cite{NS22}*{Remark 4.13} consists of elements $\gamma$ for which the $\gamma$-action on $X$ is not virtually free at $x$.
\end{proof}

\begin{rmk}\label{rmk:eq_local}
In \cite{NS22}*{Remark 4.13}, the authors proved that 
\[
\operatorname{mld}_{x'}(X/G, \mathfrak{a}) = 
\min _{\gamma \in G}
\operatorname{mld}_{x' _{\gamma}} \bigl( X/\langle \gamma \rangle , \mathfrak{a} \mathcal{O}_{X/\langle \gamma \rangle} \bigr). 
\]
Proposition \ref{prop:eq_local} asserts that this minimum is attained at some $\gamma \in \operatorname{NVF}(X,x)$.
\end{rmk}

\section{Further properties of local arc spaces}\label{section:arc_local}

In this section, we collect several properties of local arc spaces and trivial arcs, which will be used in the proofs in the subsequent sections. 
In particular, Theorem \ref{thm:5.2} plays a crucial role in the proof of Theorem \ref{thm:PIA}.

\begin{thm}\label{thm:5.2}
Suppose that a finite group $G$ acts on a normal variety $X$. 
Let $x \in X$ be a $G$-fixed closed point, and let $x' \in X' = X/G$ be its image. 
Let $\gamma \in G$ be an element. 
If $X^{[\gamma]}_{\infty, x'}$ is not a thin set of $X^{[\gamma]}_{\infty}$, then any cylinder $C \subset X^{[\gamma]}_{\infty, x'}$ containing the trivial arc (associated to $x$) is also not a thin set of $X^{[\gamma]}_{\infty}$. 
\end{thm}

\begin{proof}
Suppose that $X^{[\gamma]}_{\infty, x'}$ is not a thin set of $X^{[\gamma]}_{\infty}$. 
Since $x$ is a $G$-fixed point, it is the unique point in $X$ mapping to $x'$. 
Thus, by Theorem \ref{thm:vf_thin_local}, the $\gamma$-action on $X$ is not virtually free at $x$. 
By Proposition \ref{prop:vf_local_eq}(3), there exists a $\gamma$-equivariant proper birational morphism $f \colon Y \to X$ from a smooth variety $Y$ such that $f^{-1}(x)$ contains a $\gamma$-fixed closed point $y$. 
Because $y$ is a $\gamma$-fixed point, the trivial arc $\alpha \in Y^{[\gamma]}_{\infty}$ associated to $y$ exists. 
Then, by definition, its image $f^{[\gamma]}_{\infty} (\alpha)$ is the trivial arc of $X^{[\gamma]}_{\infty, x'}$ associated to $x$, where $f^{[\gamma]} \colon Y^{[\gamma]} \to X^{[\gamma]}$ is the $k[t]$-morphism induced by $f$. 
By assumption, the cylinder $C$ contains the trivial arc associated to $x$, which implies $f^{[\gamma]}_{\infty} (\alpha) \in C$. 
In particular, we have $\bigl( f^{[\gamma]}_{\infty} \bigr)^{-1} (C) \neq \emptyset$. 
Therefore, by Lemma \ref{lem:2.27'}, we conclude that $C$ is not a thin set of $X^{[\gamma]}_{\infty}$. 
\end{proof}

The following lemma is an immediate consequence of the definition of virtually free actions. 

\begin{lem}\label{lem:vf_properties}
Suppose that a finite group $G$ acts on a normal variety $X$. 
Let $\gamma \in G$ be an element. 
\begin{enumerate}
\item 
For any integer $\ell \ge 1$, if the $\gamma$-action on $X$ is not virtually free at a closed point $x \in X$, then the $\gamma^\ell$-action is also not virtually free at $x$. 

\item 
The locus 
\[
Z_\gamma := \{ x \in X \mid \text{the } \gamma\text{-action on } X \text{ is not virtually free at } x \}
\]
is a closed subset of $X$. 

\item
Let $D$ be a $G$-invariant normal prime divisor on $X$. If the $\gamma$-action on $D$ is not virtually free at $x$, 
then the $\gamma$-action on $X$ is also not virtually free at $x$.
\end{enumerate}
\end{lem}

\begin{proof}
Assertion (1) follows immediately from Definition \ref{defi:vf_local} and the fact that any $\gamma$-fixed point is automatically a $\gamma^\ell$-fixed point. 

Next, we prove (2). Let $f \colon Y \to X$ be a $G$-equivariant proper birational morphism from a smooth variety $Y$, 
and let $F_\gamma \subset Y$ be the $\gamma$-fixed locus. 
By Proposition \ref{prop:vf_local_eq}(3), the $\gamma$-action on $X$ is not virtually free at $x$ if and only if $f^{-1}(x)$ contains a $\gamma$-fixed point. 
This means that $Z_\gamma = f(F_\gamma)$. 
Since $F_\gamma$ is a closed subset of $Y$ and $f$ is a proper morphism, its image $f(F_\gamma)$ is a closed subset of $X$, which proves (2). 

(3) follows from Proposition \ref{prop:vf_local_eq}(2). 
\end{proof}

We introduce a natural action of the multiplicative monoid $k = k^* \cup \{0\}$ on the arc space $X^{[\gamma]}_{\infty}$. 

\begin{defi}\label{defi:k_action}
Suppose that a finite group $G$ acts on a variety $X$. 
Let $\gamma \in G$ be an element of order $d'$. 
For $c \in k$, let 
\[
\rho_c \colon \operatorname{Spec} k[[t^{1/d'}]] \to \operatorname{Spec} k[[t^{1/d'}]]
\]
be the morphism induced by the $k$-algebra endomorphism $t^{1/d'} \mapsto c t^{1/d'}$. 
Recall that a $k$-arc $\alpha \in X^{[\gamma]}_{\infty}$ corresponds to a $\gamma$-equivariant morphism $\alpha \colon \operatorname{Spec} k[[t^{1/d'}]] \to X$. 
We define the action of $c$ on $\alpha$, denoted by $c \cdot \alpha$, by the composition $\alpha \circ \rho_c$. 
We note that this action preserves the $\gamma$-equivariance since
$\rho_c$ is a $\gamma$-equivariant morphism. 
Thus, the set of $k$-arcs in $X^{[\gamma]}_{\infty}$ admits an action of the multiplicative monoid $k$. 
\end{defi}

\begin{rmk}\label{rmk:k_action}
For a $k$-scheme $Y$, the arc space $Y_\infty$ naturally admits a $k$-action by scaling the parameter $t$.
In contrast, for a general $k[t]$-scheme $Y$, the arc space $Y_{\infty}$ does not necessarily admit a $k$-action,
because the scaling of $t$ on $\operatorname{Spec} k[[t]]$ does not preserve the property of being a $k[t]$-morphism $\operatorname{Spec} k[[t]] \to Y$.
In our setting, although $X^{[\gamma]}$ is a $k[t]$-scheme, we can introduce a $k$-action on its arc space $X^{[\gamma]}_\infty$ by lifting the arcs to $\gamma$-equivariant $k$-morphisms $\operatorname{Spec} k[[t^{1/d'}]] \to X$. 
\end{rmk}

\begin{lem}\label{lem:k_action}
Suppose that a finite group $G$ acts on a variety $X$. 
Let $x \in X$ be a $G$-fixed closed point, and let $x' \in X' = X/G$ be its image. 
Let $\gamma \in G$ be an element. 
For any $k$-arc $\alpha \in X^{[\gamma]}_{\infty, x'}$, the $k$-arc $0 \cdot \alpha$ given by the $k$-action is exactly the trivial arc (associated to $x$). 
\end{lem}

\begin{proof}
Take an arbitrary $k$-arc $\alpha \in X^{[\gamma]}_{\infty, x'}$. 
Since $x$ is a $G$-fixed point, the center of $\alpha$ evaluated at $t=0$ is exactly $x$. 
By definition, the action of $0 \in k$ maps $\alpha$ to the constant morphism centered at $x$, 
which is exactly the trivial arc associated to $x$.
\end{proof}

\section{PIA conjecture for quotient singularities}\label{section:PIA}
In this section, we study the PIA conjecture for hyperquotient singularities. 
As discussed in Section \ref{section:cex}, the conjecture fails in general for this class of singularities. 
To understand this phenomenon, our main goal in this section is to establish a necessary and sufficient condition for the PIA conjecture to hold in this setting (Theorem \ref{thm:PIA_equiv}). 
As an application, we deduce Corollary \ref{cor:PIA_general}, which provides practical sufficient conditions, such as the equality $\operatorname{NVF}(X,x) = \operatorname{NVF}(D,x)$ (see Setup \ref{setup:nonlin}(2) for the definition of $\operatorname{NVF}$). 
This corollary generalizes our previous result \cite{NS25}*{Corollary 9.1} for the klt case (see Remark \ref{rmk:NVF}).

For the sake of simplicity, we first explain the case of linear actions.
\begin{setup}[Setting for Theorem \ref{thm:PIA}]\label{setup:lin}
Let $N$ be a positive integer, and let $R = k[x_1, \ldots, x_N]$. 
Let $G \subset \mathrm{GL}_N (k)$ be a finite subgroup of order $d$. 
Suppose that the linear $G$-action on $\mathbb{A}^N = \operatorname{Spec} R$ is free in codimension one. 
Let $\xi \in k$ be a primitive $d$-th root of unity. 
We define
\[
A := \mathbb{A}^N, \qquad
A' := \operatorname{Spec} R^G = A/G. 
\]
Let $x \in A$ be the origin, and let $x' \in A'$ be its image. 

Let $f_1, \ldots , f_{c} \in R^G$ be a regular sequence contained in the maximal ideal $(x_1, \ldots , x_N)^G$. 
We define 
\[
X := \operatorname{Spec} \bigl( R/(f_1, \ldots, f_{c}) \bigr) \subset A, \quad
X' := \operatorname{Spec} \bigl( R^G/(f_1, \ldots, f_{c}) \bigr) = X / G \subset A'. 
\]
Suppose that $X$ is normal. 
\end{setup}

In \cite{NS22}*{Theorem 5.1}, the following assertion was proved under the additional assumption that $X'$ is klt at $x'$. 

\begin{thm}\label{thm:PIA}
In the setting of Setup \ref{setup:lin}, 
let $\mathfrak{a} \subset R^G$ be a non-zero ideal, and let $\tau$ be a non-negative real number. 
Suppose that $\mathfrak{b} := \mathfrak{a} \mathcal{O}_{X'} \neq 0$. 
Suppose that the $\gamma$-action on $X$ is not virtually free at $x$ for all $\gamma \in G$. 
Then we have 
\[
\operatorname{mld}_{x'} \bigl( A', (f_1, \ldots, f_c) \mathfrak{a}^{\tau} \bigr) 
= \operatorname{mld}_{x'} \bigl( X', \mathfrak{b}^{\tau} \bigr). 
\]
\end{thm}

\begin{proof}
In \cite{NS22}*{Theorem 5.1}, the corresponding assertion was proved under the stronger assumption that $X'$ is klt at $x'$, instead of the condition that the $\gamma$-action on $X$ is not virtually free at $x$ for all $\gamma \in G$. 
While we replace the klt assumption with our new condition, the basic structure of the proof remains parallel to theirs. 
Here, Theorem \ref{thm:5.2} plays a key role in making this replacement possible, so we detail only the necessary modifications below. 

For integers $w, v \ge 0$ and an element $\gamma \in G$, we set
\[
C'_{w,v,\gamma} := 
\operatorname{Cont}^{\ge w} \bigl( (f_1 \cdots f_c) \bigr) \cap
\operatorname{Cont}^{\ge v} (\mathfrak{a}) \cap
\operatorname{Cont}^{\ge 1} (\mathfrak{m}_{x'}) \subset A'_{\infty}. 
\]
By definition, $C'_{w,v,\gamma}$ is a closed cylinder of $A'_{\infty}$.

Let $A^{(\gamma)}_{\infty}$, $A^{[\gamma]}_{\infty}$, their closed subschemes $X^{(\gamma)}_{\infty}$, $X^{[\gamma]}_{\infty}$, and the morphisms $\lambda_{\gamma, \infty}$, $h_{\infty}$ be as defined in Sections \ref{section:vfc}, \ref{section:dl1}, and \ref{section:dl2}. 
Then, we have the following commutative diagram:
\[
\xymatrix@R=1.5em@C=3em{
A^{(\gamma)}_{\infty} \ar[r]^-{\lambda_{\gamma, \infty}} & A^{[\gamma]}_{\infty} \ar[r]^-{h_{\infty}} & A'_{\infty} \\
X^{(\gamma)}_{\infty} \ar[r] \ar@{}[u]|-{\rotatebox{90}{$\subset$}} & X^{[\gamma]}_{\infty} \ar[r] \ar@{}[u]|-{\rotatebox{90}{$\subset$}} & X'_{\infty} \ar@{}[u]|-{\rotatebox{90}{$\subset$}}
}
\]

In our notation, the set denoted by $C_{w,v,\gamma}$ in the proof of \cite{NS22}*{Theorem 5.1} corresponds to the inverse image of $C'_{w,v,\gamma}$ in $A^{(\gamma)}_{\infty}$. 
Thus, we define 
\[
C_{w,v,\gamma} := (h_{\infty} \circ \lambda_{\gamma, \infty})^{-1} \bigl( C'_{w,v,\gamma} \bigr). 
\]

The key step in the proof of \cite{NS22}*{Theorem 5.1} is their Claim 5.2, which translates to our setting as follows: 
\begin{claim}[cf.\ \cite{NS22}*{Claim 5.2}]\label{claim:1}
For every irreducible component $C$ of $C_{w,v,\gamma}$, the intersection $C \cap X^{(\gamma)}_{\infty}$ is not a thin set of $X^{(\gamma)}_{\infty}$. 
\end{claim}

By Propositions \ref{prop:()to[]} and \ref{prop:thin_eq2}, Claim \ref{claim:1} reduces to the following claim: 
\begin{claim}\label{claim:2}
For every irreducible component $C$ of $h_{\infty}^{-1} \bigl( C'_{w,v,\gamma} \bigr)$, 
the intersection $C \cap X^{[\gamma]}_{\infty}$ is not a thin set of $X^{[\gamma]}_{\infty}$. 
\end{claim}

\begin{proof}[Proof of Claim \ref{claim:2}]
Since the $\gamma$-action is not virtually free at $x$ by assumption, 
the cylinder $X^{[\gamma]}_{\infty, x'}$ is not a thin set of $X^{[\gamma]}_{\infty}$ by Theorem \ref{thm:vf_thin_local}. 
By Theorem \ref{thm:5.2}, in order to show the claim, it is sufficient to show that $C \cap X^{[\gamma]}_{\infty}$ contains the trivial arc of $X^{[\gamma]}_{\infty, x'}$. 
Since the trivial arc of $X^{[\gamma]}_{\infty, x'}$ coincides with that of $A^{[\gamma]}_{\infty, x'}$, 
it is sufficient to show that $C$ contains the trivial arc of $A^{[\gamma]}_{\infty, x'}$. 

Recall that the $k^*$-action on $A^{[\gamma]}_{\infty}$ introduced in Definition \ref{defi:k_action} is defined by the reparametrization $t^{1/d'} \mapsto c t^{1/d'}$. 
By equipping $A'_{\infty}$ with the corresponding $k^*$-action defined by $t \mapsto c^{d'} t$, the induced morphism $h_{\infty}$ is $k^*$-equivariant. 
Since the contact loci are defined by $t$-adic orders, they are invariant under any reparametrization. 
Thus, the closed cylinder $C'_{w,v,\gamma} \subset A'_{\infty}$ is invariant under the $k^*$-action, which implies that its inverse image $h_{\infty}^{-1} \bigl( C'_{w,v,\gamma} \bigr)$ is invariant under the $k^*$-action on $A^{[\gamma]}_{\infty}$. 
Since $k^*$ is connected, its irreducible component $C$ is also $k^*$-invariant. 
Then, Lemma \ref{lem:k_action} and the closedness of $C$ guarantee that it contains the trivial arc of $A^{[\gamma]}_{\infty, x'}$. 
\end{proof}

This establishes Claim \ref{claim:1}. The rest of the proof of \cite{NS22}*{Theorem 5.1} now goes through without any changes. 
This concludes the proof of the theorem.
\end{proof}

Theorem \ref{thm:PIA} remains valid for general finite group actions, not necessarily restricted to linear actions.

\begin{setup}[General setting]\label{setup:nonlin}
\begin{enumerate}
\item
Suppose that a finite group $G$ acts on a smooth variety $A$, and that the action is free in codimension one. 
Let $A' := A / G$ be the quotient variety. 
Let $x \in A$ be a $G$-fixed closed point, and let $x' \in A'$ be its image. 
Let $X' \subset A'$ be a normal subvariety of codimension $c$, locally defined at $x'$ by $c$ equations $f_1, \ldots, f_c \in \mathfrak{m}_{A', x'}$. 
Let $X$ be the inverse image of $X'$ under the quotient morphism $A \to A'$. 

\item
Let $D'$ be a normal prime Cartier divisor on $X'$ passing through $x'$, and let $D$ be its inverse image in $X$. 
Let $\operatorname{NVF}(X,x)$ (resp.\ $\operatorname{NVF}(D,x)$) denote the set of elements $\gamma \in G$ such that the $\gamma$-action on $X$ (resp.\ $D$) is not virtually free at $x$. 

\item
Let $\mathfrak{b} \subset \mathcal{O}_{X'}$ be an $\mathbb{R}$-ideal sheaf. 
Suppose that $D'$ is not contained in the cosupport of the $\mathbb{R}$-ideal sheaf $\mathfrak{b}$. 
\end{enumerate}
\end{setup}

\begin{thm}\label{thm:PIA_nonlin}
In the setting of Setup \ref{setup:nonlin}(1), 
let $\mathfrak{a} \subset \mathcal{O}_{A'}$ be an ideal sheaf and let $\tau$ be a positive real number. 
Suppose that $\mathfrak{b} := \mathfrak{a} \mathcal{O}_{X'} \neq 0$. 
Suppose that the $\gamma$-action on $X$ is not virtually free at $x$ for all $\gamma \in G$. 
Then we have
\[
\operatorname{mld}_{x'} \bigl(A', (f_1 \cdots f_c) \mathfrak{a} ^{\tau} \bigr) 
= \operatorname{mld}_{x'} (X',\mathfrak{b}^{\tau}). 
\]
\end{thm}

\begin{proof}
This assertion can be proved by combining the arguments in the proofs of Theorem \ref{thm:PIA} and \cite{NS25}*{Theorem 8.2}. 
We give an outline of the proof below. 
As before, we set $R := k[x_1, \ldots, x_N]$. 

By the argument at the beginning of the proof of \cite{NS25}*{Theorem 8.2}, there exist a linear $G$-action on $\widehat{R} := k[[x_1, \ldots , x_N]]$ and a $G$-equivariant \'{e}tale homomorphism $\mathcal{O}_{A,x} \to \widehat{R}$.

We fix $\gamma \in G$. 
Take a new basis $x_1^{(\gamma)}, \ldots, x_N^{(\gamma)}$ consisting of $k$-linear combinations of the original basis $x_1, \ldots , x_N$ such that the action of $\gamma$ is represented by the diagonal matrix $\operatorname{diag} \left( \xi ^{e_1}, \ldots , \xi ^{e_N} \right)$ with respect to this basis, where $1 \le e_i \le d$. 

We define a homomorphism $\lambda^* _{\gamma} \colon \widehat{R} \to R[[t^{1/d}]]$ by $\lambda^* _{\gamma} (x_i^{(\gamma)}) = t^{e_i/d} x_i^{(\gamma)}$. 
Restricting the composition $\mathcal{O}_{A, x} \to \widehat{R} \to R[[t^{1/d}]]$ to $\mathcal{O}_{A, x} ^G = \mathcal{O}_{A', x'}$, we obtain a homomorphism $\mathcal{O}_{A', x'} \to R[[t]]$. 
We define $\widehat{A}^{(\gamma)} := \operatorname{Spec} R[[t]]$, and let $\widehat{A}^{(\gamma)}_{\infty}$ denote its arc space as a $k[t]$-scheme. 
This homomorphism induces a morphism $\widehat{A}^{(\gamma)}_{\infty} \to A'_{\infty, x'}$. 

Let $\widehat{X}^{(\gamma)}$ be the subscheme of $\widehat{A}^{(\gamma)}$ defined by the images $F_1, \ldots , F_c \in R[[t]]$ of $f_1, \ldots , f_c$ under the homomorphism $\mathcal{O}_{A', x'} \to R[[t]]$. Then, we have the following diagram: 
\[
\xymatrix@R=1.5em@C=2em{
\widehat{A}^{(\gamma)}_{\infty} \ar[r] & A'_{\infty, x'} \\
\widehat{X}^{(\gamma)}_{\infty} \ar[r] \ar@{}[u]|-{\rotatebox{90}{$\subset$}} & X'_{\infty, x'} \ar@{}[u]|-{\rotatebox{90}{$\subset$}}
}
\]

Let $d'$ be the order of $\gamma$. 
Notice that the homomorphism $\mathcal{O}_{A, x} \to R[[t^{1/d}]]$ constructed above naturally factors through the polynomial rings obtained by adjoining $t^{1/d'}$ as follows: 
\[
\mathcal{O}_{A, x} \hookrightarrow \mathcal{O}_{A, x}[t^{1/d'}] \hookrightarrow \widehat{R}[t^{1/d'}] \xrightarrow{\lambda^* _{\gamma}} R[[t^{1/d}]]. 
\]
By taking the invariant subrings, we obtain a sequence of homomorphisms 
\[
\mathcal{O}_{A', x'} \to (\mathcal{O}_{A,x}[t^{1/d'}])^{\langle \gamma \rangle} \to R[[t]], 
\]
which geometrically induces the sequence of morphisms 
\[
\widehat{A}^{(\gamma)}_{\infty} \to A^{[\gamma]}_{\infty, x'} \to A'_{\infty, x'}. 
\]
Therefore, we obtain the following commutative diagram: 
\[
\xymatrix@R=1.5em@C=3em{
\widehat{A}^{(\gamma)}_{\infty} \ar[r]^-{\lambda_{\gamma, \infty}} & A^{[\gamma]}_{\infty, x'} \ar[r]^-{h_{\infty}} & A'_{\infty, x'} \\
\widehat{X}^{(\gamma)}_{\infty} \ar[r] \ar@{}[u]|-{\rotatebox{90}{$\subset$}} & X^{[\gamma]}_{\infty, x'} \ar[r] \ar@{}[u]|-{\rotatebox{90}{$\subset$}} & X'_{\infty, x'} \ar@{}[u]|-{\rotatebox{90}{$\subset$}}
}
\]
By the same argument as in Proposition \ref{prop:()to[]}, $\lambda_{\gamma, \infty}$ induces a natural homeomorphism $\widehat{A}^{(\gamma)}_{\infty}(K) \xrightarrow{\sim} A^{[\gamma]}_{\infty, x'}(K)$ for any field extension $K$ of $k$.

For integers $w, v \ge 0$ and an element $\gamma \in G$, we set
\[
C'_{w,v,\gamma} := 
\operatorname{Cont}^{\ge w} \bigl( (f_1 \cdots f_c) \bigr) \cap
\operatorname{Cont}^{\ge v} (\mathfrak{a}) \cap
\operatorname{Cont}^{\ge 1} (\mathfrak{m}_{x'}) \subset A'_{\infty, x'}. 
\]
We also define a cylinder $C_{w,v,\gamma} \subset \widehat{A}^{(\gamma)}_{\infty}$ by
\[
C_{w,v,\gamma} := (h_{\infty} \circ \lambda_{\gamma, \infty})^{-1} \bigl( C'_{w,v,\gamma} \bigr). 
\]

The key step in the proof of \cite{NS25}*{Theorem 8.2} is their Claim 8.3, which translates to our setting as follows: 
\begin{claim}[cf.\ \cite{NS25}*{Claim 8.3}]\label{claim:3}
Let $\operatorname{Jac}'_{\widehat{X}^{(\gamma)}/k[[t]]}$ denote the ideal sheaf on $\widehat{X}^{(\gamma)}$ generated by $c$-minors of the Jacobian matrix $\bigl( \partial F_i/\partial x_j \bigr)_{\substack{1 \le i \le c, \\ 1 \le j \le N}}$. 
Let $\widehat{H} \subset \widehat{X}^{(\gamma)}$ be the closed subscheme defined by $\operatorname{Jac}'_{\widehat{X}^{(\gamma)}/k[[t]]}$.
For every irreducible component $C$ of $C_{w,v,\gamma}$, we have 
\[
C \cap \widehat{X}^{(\gamma)}_{\infty} \not \subset \widehat{H}_{\infty}.
\] 
\end{claim}
\begin{proof}[Proof of Claim \ref{claim:3}]
By Claim \ref{claim:2}, 
$\lambda_{\gamma, \infty}(C) \cap X^{[\gamma]}_{\infty, x'}$ is not a thin set of $X^{[\gamma]}_{\infty}$. 
Then by definition, we have $\lambda_{\gamma, \infty}(C) \cap X^{[\gamma]}_{\infty, x'} \not \subset H_{\infty}$, where $H \subset X^{[\gamma]}$ is the closed subscheme defined by the relative Jacobian ideal $\operatorname{Jac}_{X^{[\gamma]}/k[t]}$. 
By the same observation as in Remark \ref{rmk:iso_out0}, we have 
\[
\operatorname{Jac}_{X^{[\gamma]}/k[t]} \mathcal{O}_{\widehat{X}^{(\gamma)}} = \operatorname{Jac}'_{\widehat{X}^{(\gamma)}/k[[t]]}
\]
outside $t = 0$. 
Therefore, the arc spaces $\widehat{H}_{\infty}$ and $H_{\infty}$ correspond to each other under  $\lambda_{\gamma, \infty}$ (cf.\ Remark \ref{rmk:dom}). 
Thus, the condition $\lambda_{\gamma, \infty}(C) \cap X^{[\gamma]}_{\infty, x'} \not \subset H_{\infty}$ implies that $C \cap \widehat{X}^{(\gamma)}_{\infty} \not \subset \widehat{H}_{\infty}$.
\end{proof}
The rest of the proof of \cite{NS25}*{Theorem 8.2} works without any changes. 
\end{proof}

\begin{rmk}
Theorem \ref{thm:PIA_nonlin} also holds for an arbitrary $\mathbb{R}$-ideal $\mathfrak{a}$ on $A'$. 
That is, we have
\[
\operatorname{mld}_{x'} \bigl(A', (f_1 \cdots f_c) \mathfrak{a} \bigr) 
= \operatorname{mld}_{x'} (X',\mathfrak{b})
\]
for $\mathfrak{b} := \mathfrak{a} \mathcal{O}_{X'}$. 
\end{rmk}

As shown in \cite{NS4}, there exist counterexamples to the PIA conjecture under the conditions of Setup \ref{setup:nonlin}(1)(2)(3) (see Example \ref{eg:counter_ex}). 
In this setting, we establish a necessary and sufficient condition for the PIA conjecture to hold (Theorem \ref{thm:PIA_equiv}). 
We begin with the following remark. 

\begin{rmk}\label{rmk:PIA_eq}
In the setting of Setup \ref{setup:nonlin}(1)(2)(3), it follows from \cite{NS22}*{Remark 4.13} that 
\begin{equation}
\operatorname{mld}_{x'} \bigl( X', \mathfrak{b} \mathcal{O}_{X'} (-D') \bigr) = 
\min _{\gamma \in G} \operatorname{mld}_{x'_{\gamma}} \bigl( X/\langle \gamma \rangle, \mathfrak{b} \mathcal{O}_{X/\langle \gamma \rangle} (-D_{\gamma}) \bigr), 
\tag{$\spadesuit$}\label{eq:spadesuit}
\end{equation}
where $x'_{\gamma}$ and $D_{\gamma}$ denote the images of $x$ and $D$ in $X/\langle \gamma \rangle$, respectively. 
By Proposition \ref{prop:eq_local}, the minimum in \eqref{eq:spadesuit} is attained at some $\gamma \in \operatorname{NVF}(X,x)$. 
We also note that the inclusion $\operatorname{NVF}(D,x) \subset \operatorname{NVF}(X,x)$ holds by Lemma \ref{lem:vf_properties}(3). 
\end{rmk}

\begin{thm}\label{thm:PIA_equiv}
In the setting of Setup \ref{setup:nonlin}(1)(2)(3) and Remark \ref{rmk:PIA_eq}, 
the following conditions are equivalent: 
\begin{enumerate}
\item[(i)]
The minimum in \eqref{eq:spadesuit} is attained at some $\gamma \in \operatorname{NVF}(D,x)$. 

\item[(ii)]
The PIA conjecture holds; that is, 
\[
\operatorname{mld}_{x'} \bigl( X', \mathfrak{b} \mathcal{O}_{X'} (-D') \bigr) = 
\operatorname{mld}_{x'} (D', \mathfrak{b} \mathcal{O}_{D'}). 
\]
\end{enumerate}
\end{thm}

\begin{proof}
We first show that for any $\gamma \in \operatorname{NVF}(D,x)$, we have 
\begin{equation}\label{eq:term_eq}
\operatorname{mld}_{x'_{\gamma}} \bigl( X/\langle \gamma \rangle, \mathfrak{b} \mathcal{O}_{X/\langle \gamma \rangle} (-D_{\gamma}) \bigr)
= 
\operatorname{mld}_{x'_{\gamma}} \bigl( D/\langle \gamma \rangle, \mathfrak{b} \mathcal{O}_{D/\langle \gamma \rangle} \bigr). 
\end{equation}
Fix an element $\gamma \in \operatorname{NVF}(D,x)$. 
By Lemma \ref{lem:vf_properties}(1) and the inclusion $\operatorname{NVF}(D,x) \subset \operatorname{NVF}(X,x)$ from Lemma \ref{lem:vf_properties}(3), the action of every element in $\langle \gamma \rangle$ is not virtually free on both $D$ and $X$ at $x$. 
This justifies the application of Theorem \ref{thm:PIA_nonlin} to the $\langle \gamma \rangle$-actions on $X$ and $D$. 
Recall that $X'$ is locally defined by $f_1, \ldots , f_c \in \mathcal{O}_{A',x'}$ in $A'$ at $x'$. 
Take an element $g \in \mathcal{O}_{A',x'}$ such that its image $\overline{g} \in \mathcal{O}_{X',x'}$ locally defines $D'$, and take an $\mathbb{R}$-ideal sheaf $\mathfrak{a}$ on $A'$ such that $\mathfrak{b} = \mathfrak{a} \mathcal{O}_{X'}$. 
We may naturally regard $f_1, \ldots, f_c$ and $g$ as elements of $\mathcal{O}_{A/\langle \gamma \rangle, x'_{\gamma}}$. 
Then $f_1, \ldots, f_c$ define $X/\langle \gamma \rangle$ in $A/\langle \gamma \rangle$ locally at $x'_{\gamma}$, and the image of $g$ in $\mathcal{O}_{X/\langle \gamma \rangle, x'_{\gamma}}$ defines $D_{\gamma}$. 
Note also that $\mathfrak{b} \mathcal{O}_{X/\langle \gamma \rangle} = \mathfrak{a} \mathcal{O}_{X/\langle \gamma \rangle}$. 
By applying Theorem \ref{thm:PIA_nonlin} twice, we obtain
\begin{align*}
\operatorname{mld}_{x'_{\gamma}} \bigl( X/\langle \gamma \rangle, \mathfrak{b} \mathcal{O}_{X/\langle \gamma \rangle} (-D_{\gamma}) \bigr) 
&= \operatorname{mld}_{x'_{\gamma}} \bigl( A/\langle \gamma \rangle, (f_1 \cdots f_c \cdot g) \mathfrak{a} \mathcal{O}_{A/\langle \gamma \rangle} \bigr) \\
&= \operatorname{mld}_{x'_{\gamma}} \bigl( D/\langle \gamma \rangle, \mathfrak{b} \mathcal{O}_{D/\langle \gamma \rangle} \bigr), 
\end{align*}
which proves \eqref{eq:term_eq}. 

Now we prove the equivalence (i) $\Leftrightarrow$ (ii). 
By applying Proposition \ref{prop:eq_local} to the $G$-action on $D$, we have 
\[
\operatorname{mld}_{x'} ( D', \mathfrak{b} \mathcal{O}_{D'} ) = 
\min _{\gamma \in \operatorname{NVF}(D,x)} \operatorname{mld}_{x'_{\gamma}} \bigl( D/\langle \gamma \rangle, \mathfrak{b} \mathcal{O}_{D/\langle \gamma \rangle} \bigr). 
\]
Using equation \eqref{eq:term_eq}, we can rewrite this as 
\[
\operatorname{mld}_{x'} ( D', \mathfrak{b} \mathcal{O}_{D'} ) = 
\min _{\gamma \in \operatorname{NVF}(D,x)} \operatorname{mld}_{x'_{\gamma}} \bigl( X/\langle \gamma \rangle, \mathfrak{b} \mathcal{O}_{X/\langle \gamma \rangle} (-D_{\gamma}) \bigr). 
\]
On the other hand, by Proposition \ref{prop:eq_local}, we have
\[
\operatorname{mld}_{x'} \bigl( X', \mathfrak{b} \mathcal{O}_{X'} (-D') \bigr) = 
\min _{\gamma \in \operatorname{NVF}(X,x)} \operatorname{mld}_{x'_{\gamma}} \bigl( X/\langle \gamma \rangle, \mathfrak{b} \mathcal{O}_{X/\langle \gamma \rangle} (-D_{\gamma}) \bigr). 
\]
Since $\operatorname{NVF}(D,x) \subset \operatorname{NVF}(X,x)$, these two minimums coincide if and only if condition (i) holds. 
This completes the proof.
\end{proof}

\begin{cor}\label{cor:PIA_general}
In the setting of Setup \ref{setup:nonlin}(1)(2)(3),  
suppose that one of the following conditions holds: 
\begin{enumerate}
\item[(i)] $\operatorname{NVF}(X, x) = \operatorname{NVF}(D, x)$. 
\item[(ii)] $\operatorname{NVF}(D, x) = G$, that is, for every $\gamma \in G$, the $\gamma$-action on $D$ is not virtually free at $x$.
\end{enumerate}
Then the PIA conjecture holds; that is, 
\[
\operatorname{mld}_{x'} \bigl( X', \mathfrak{b} \mathcal{O}_{X'} (-D') \bigr) = 
\operatorname{mld}_{x'} (D', \mathfrak{b} \mathcal{O}_{D'}). 
\]
\end{cor}
\begin{proof}
If condition (i) holds, the assertion follows from Theorem \ref{thm:PIA_equiv}. 
If condition (ii) holds, the inclusions $\operatorname{NVF}(D,x) \subset \operatorname{NVF}(X,x) \subset G$ implies that $\operatorname{NVF}(X, x) = \operatorname{NVF}(D, x)$, which reduces to condition (i). 
\end{proof}

\begin{rmk}\label{rmk:NVF}
In \cite{NS25}, the authors proved the assertion of Corollary \ref{cor:PIA_general} under the assumption that $D'$ is klt at $x'$ (see \cite{NS25}*{Corollary 9.1}).
This assumption implies that $\operatorname{NVF}(D, x) = G_x = G$ by Theorem \ref{thm:klt_local}, thereby satisfying condition (ii).
\end{rmk}

\begin{rmk}
Theorem \ref{thm:PIA} and Corollary \ref{cor:PIA_general} can be extended to the more general setting 
where the defining equations $f_1, \ldots, f_c$ of $X$ in $A$ are $G$-semi-invariant, as treated in \cite{NS26}*{Theorems 7.1 and 7.3}. 
\end{rmk}

\section{Remarks on counterexamples to the PIA conjecture} \label{section:cex}
In \cite{NS4}, the authors give a counterexample to the PIA conjecture. 
Here, we re-examine this example from the viewpoint of Theorem \ref{thm:PIA_equiv}. 

\begin{eg}[\cite{NS4}*{Example 1.4}]\label{eg:counter_ex}
Let $\xi \in k$ be a primitive cube root of unity.
Consider the polynomial ring $R := k[x_1, x_2, x_3, x_4, x_5]$ and the polynomial 
\[
f := x_1^3 + x_2^3 + x_3^3 \in R. 
\]
Let $G := \langle \gamma \rangle$ be a cyclic group of order $3$, where the generator $\gamma$ acts on the affine space $A := \operatorname{Spec}(R) = \mathbb{A}^5_k$ via the diagonal matrix 
\[
\gamma := \operatorname{diag}(1, \xi, \xi^2, \xi, \xi). 
\]
Since the polynomial $f$ is invariant under this $G$-action, the action naturally restricts to the hypersurface $X := \operatorname{Spec}(R/(f))$. 
Let 
\[
A' := A/G \quad \text{and} \quad X' := X/G
\]
be the corresponding quotient varieties, and let $x' \in A'$ denote the image of the origin of $A$. 
By \cite{NS4}*{Theorem 1.5, Appendix A}, we have 
\[
\operatorname{mld}_{x'}(A', X') = \frac{5}{3} \quad \text{and} \quad 
\operatorname{mld}_{x'}(X') = 2. 
\]
Therefore, we obtain 
\[
\operatorname{mld}_{x'}(A', X') < \operatorname{mld}_{x'}(X'), 
\]
which implies that the pair $(A', X')$ does not satisfy the PIA conjecture.
\end{eg}

We now revisit this example from the viewpoint of Theorem \ref{thm:PIA_equiv}.
In this example, we have
\[
\operatorname{NVF}(A,x) = G = \langle \gamma \rangle, \qquad
\operatorname{NVF}(X,x) = \{ 1_G \}. 
\]
One can also see that 
\[
\operatorname{mld}_{x}(A, X) = \operatorname{mld}_{x}(X) = 2, \qquad \operatorname{mld}_{x'}(A', X') = \frac{5}{3}. 
\]
This implies that the minimum in 
\[
\operatorname{mld}_{x'}(A', X') = \min _{\gamma' \in G} \operatorname{mld}_{x'_{\gamma '}} \bigl( A/\langle \gamma' \rangle, X/\langle \gamma' \rangle \bigr)
\]
is strictly less than the value at $\gamma' = 1_G$ (which is $\operatorname{mld}_{x}(A, X) = 2$). 
Therefore, the minimum is not attained at $\gamma' = 1_G$. 
Since $\operatorname{NVF}(X,x) = \{ 1_G \}$, the minimum is not attained at any element in $\operatorname{NVF}(X,x)$. 
By Theorem \ref{thm:PIA_equiv}, this failure of condition (i) explains why the PIA conjecture does not hold for this pair.

\section{LSC conjecture for hyperquotient singularities}\label{section:LSC}

In this section, we prove the LSC conjecture for hyperquotient singularities $Y'$ (see Theorem \ref{thm:LSC}). 
Note that $Y'$ is assumed to be klt in \cite{NS25}*{Theorem 9.2}.

\begin{thm}\label{thm:LSC}
Suppose that a finite subgroup $G$ acts on a smooth variety $X$. 
Let $X' := X / G$ be the quotient variety. 
Let $Y'$ be a normal subvariety of $X'$ of codimension $c$, and 
let $\mathfrak{a}$ be a non-zero $\mathbb{R}$-ideal sheaf on $Y'$. 
Suppose that $Y'$ is locally defined by $c$ equations in $X'$. 
Then the function 
\[
|Y'|_{\rm cl} \to \mathbb{R}_{\ge 0} \cup \{ - \infty \}; \quad x' \mapsto \operatorname{mld}_{x'}(Y',\mathfrak{a})
\]
is lower semi-continuous, where we denote by $|Y'|_{\rm cl}$ the set of all closed points of $Y'$ with the Zariski topology. 
\end{thm}

\begin{proof}
First, we reduce the problem to the case where the $G$-action on $X$ is free in codimension one. 
Let $x \in X$ be a closed point, and let $x'$ denote its image in the quotient variety. 
Let $G_x \subset G$ be the stabilizer of $x$, and let $G_x ^{\rm pr} \subset G_x$ be the normal subgroup generated by the pseudo-reflections in $G_x$. 
Then, the natural morphism $X/G_x \to X/G$ is \'{e}tale in an open neighborhood of $x'$. 
Furthermore, we have a natural isomorphism $X/G_x \simeq (X/G_x ^{\rm pr})/(G_x/G_x ^{\rm pr})$. 
Since $G_x ^{\rm pr}$ is generated by pseudo-reflections, it follows from the Chevalley--Shephard--Todd theorem that the quotient variety $X/G_x ^{\rm pr}$ is smooth in an open neighborhood of $x'$. 
Therefore, by replacing $X$ with this open neighborhood of $x'$ in $X/G_x ^{\rm pr}$, and $G$ with the quotient group $G_x/G_x ^{\rm pr}$, 
we may assume without loss of generality that the $G$-action on $X$ is free in codimension one. 

Take an ideal $\mathfrak{b} \subset \mathcal{O}_{X'}$ such that $\mathfrak{b} \mathcal{O}_{Y'} = \mathfrak{a}$. 
Let $Y \subset X$ be the inverse image of $Y' \subset X'$. 

For each $\gamma \in G$, let $Z_\gamma \subset Y$ be the closed subset of points where the $\gamma$-action on $Y$ is not virtually free (see Lemma \ref{lem:vf_properties}(2)). 
Since the quotient map $Y \to Y'$ is finite, its image $Z'_\gamma$ in $Y'$ is also closed. 
We define a function $g_\gamma \colon |Y'|_{\rm cl} \to \mathbb{R}_{\ge 0} \cup \{ - \infty, \infty \}$ by 
\[
g_\gamma(x') := 
\begin{cases}
\operatorname{mld}_{x' _{\gamma}} \bigl(X/\langle \gamma \rangle , (f_1 \cdots f_c) \mathfrak{b}  \mathcal{O}_{X/\langle \gamma \rangle} \bigr) & \text{if } x' \in Z'_\gamma, \\
\infty & \text{if } x' \notin Z'_\gamma,
\end{cases}
\]
where $f_1, \ldots, f_c \in \mathcal{O}_{X',x'}$ are local equations of $Y'$ in $X'$, and $x'_{\gamma}$ denotes the image of $x'$ in $X/\langle \gamma \rangle$. 
By \cite{Nak16}*{Corollary 1.3}, the minimal log discrepancy on a quotient of a smooth variety is lower semi-continuous. 
Since $Z'_\gamma$ is a closed subset, the extended function $g_\gamma$ is lower semi-continuous on the whole $Y'$. 

Let $x \in Y$ be a closed point and let $x'$ be its image in $Y'$. 
Let $S_x$ be the set of elements $\gamma \in G$ such that the $\gamma$-action on $Y$ is not virtually free at $x$.  
By Proposition \ref{prop:eq_local} and Theorem \ref{thm:PIA_nonlin}, we have 
\begin{align*}
\operatorname{mld}_{x'}(Y', \mathfrak{a}) 
&= \min _{\gamma \in S_x} \operatorname{mld}_{x' _{\gamma}} (Y/\langle \gamma \rangle , \mathfrak{a} \mathcal{O}_{Y/\langle \gamma \rangle}) \\
&= \min _{\gamma \in S_x} \operatorname{mld}_{x' _{\gamma}} \bigl(X/\langle \gamma \rangle , (f_1 \cdots f_c) \mathfrak{b}  \mathcal{O}_{X/\langle \gamma \rangle} \bigr). 
\end{align*}
Here, when applying Theorem \ref{thm:PIA_nonlin} for the second equality, we replace $\langle \gamma \rangle$ with its stabilizer subgroup at $x$ if necessary, because the theorem assumes that the point $x$ is fixed by the group action. 
Lemma \ref{lem:vf_properties}(1) is also used to ensure that for each $\gamma \in S_x$, the action of every element in $\langle \gamma \rangle$ is also not virtually free at $x$, which justifies the application of Theorem \ref{thm:PIA_nonlin} to the quotient variety $X/\langle \gamma \rangle$. 

In particular, we have 
\[
\operatorname{mld}_{x'}(Y', \mathfrak{a})
= \min _{\gamma \in S} \operatorname{mld}_{x' _{\gamma}} \bigl(X/\langle \gamma \rangle , (f_1 \cdots f_c) \mathfrak{b}  \mathcal{O}_{X/\langle \gamma \rangle} \bigr), 
\]
where $S := \bigcup _{x \in q^{-1}(x')} S_x$, and $q \colon Y \to Y'$ is the quotient morphism. 
Note that by definition, $\gamma \in S_x$ if and only if $x \in Z_{\gamma}$.
Therefore, $\gamma \in S$ if and only if $x' \in Z' _{\gamma}$. 
Thus, we obtain
\begin{align*}
\operatorname{mld}_{x'}(Y', \mathfrak{a}) 
&= \min _{\gamma \in S} \operatorname{mld}_{x' _{\gamma}} \bigl(X/\langle \gamma \rangle , (f_1 \cdots f_c) \mathfrak{b}  \mathcal{O}_{X/\langle \gamma \rangle} \bigr) \\
&= \min _{\substack{\gamma \in G; \\ x' \in Z'_{\gamma}}} \operatorname{mld}_{x' _{\gamma}} \bigl(X/\langle \gamma \rangle , (f_1 \cdots f_c) \mathfrak{b}  \mathcal{O}_{X/\langle \gamma \rangle} \bigr) \\
&= \min_{\gamma \in G} g_\gamma(x'). 
\end{align*}
Since each $g_{\gamma}$ is lower semi-continuous, their finite minimum $x' \mapsto \operatorname{mld}_{x'}(Y', \mathfrak{a})$ is also lower semi-continuous on $Y'$. 
This completes the proof. 
\end{proof}

\begin{rmk}
Theorem \ref{thm:LSC} can be extended to the more general setting where the defining equations $f_1, \ldots, f_c$ of $Y$ in $X$ are $G$-semi-invariant, as treated in \cite{NS26}*{Theorem 7.4}. 
\end{rmk}


\begin{bibdiv}
\begin{biblist*}

\bib{92}{collection}{
   title={Flips and abundance for algebraic threefolds},
   note={Papers from the Second Summer Seminar on Algebraic Geometry held at
   the University of Utah, Salt Lake City, Utah, August 1991;
   Ast\'{e}risque No. 211 (1992) (1992)},
   publisher={Soci\'{e}t\'{e} Math\'{e}matique de France, Paris},
   date={1992}, 
}

\bib{Amb99}{article}{
   author={Ambro, Florin},
   title={On minimal log discrepancies},
   journal={Math. Res. Lett.},
   volume={6},
   date={1999},
   number={5-6},
   pages={573--580},
}


\bib{CLNS}{book}{
   author={Chambert-Loir, Antoine},
   author={Nicaise, Johannes},
   author={Sebag, Julien},
   title={Motivic integration},
   series={Progress in Mathematics},
   volume={325},
   publisher={Birkh\"{a}user/Springer, New York},
   date={2018},
}


\bib{DL02}{article}{
   author={Denef, Jan},
   author={Loeser, Fran\c{c}ois},
   title={Motivic integration, quotient singularities and the McKay
   correspondence},
   journal={Compositio Math.},
   volume={131},
   date={2002},
   number={3},
   pages={267--290},
}


\bib{EM04}{article}{
   author={Ein, Lawrence},
   author={Musta{\c{t}}{\v{a}}, Mircea},
   title={Inversion of adjunction for local complete intersection varieties},
   journal={Amer. J. Math.},
   volume={126},
   date={2004},
   number={6},
   pages={1355--1365},
}

\bib{EM09}{article}{
   author={Ein, Lawrence},
   author={Musta\c{t}\u{a}, Mircea},
   title={Jet schemes and singularities},
   conference={
      title={Algebraic geometry---Seattle 2005. Part 2},
   },
   book={
      series={Proc. Sympos. Pure Math.},
      volume={80},
      publisher={Amer. Math. Soc., Providence, RI},
   },
   date={2009},
   pages={505--546},
}

\bib{EMY03}{article}{
   author={Ein, Lawrence},
   author={Musta{\c{t}}{\u{a}}, Mircea},
   author={Yasuda, Takehiko},
   title={Jet schemes, log discrepancies and inversion of adjunction},
   journal={Invent. Math.},
   volume={153},
   date={2003},
   number={3},
   pages={519--535},
}




\bib{GHS03}{article}{
   author={Graber, Tom},
   author={Harris, Joe},
   author={Starr, Jason},
   title={Families of rationally connected varieties},
   journal={J. Amer. Math. Soc.},
   volume={16},
   date={2003},
   number={1},
   pages={57--67},
}

\bib{HM07}{article}{
   author={Hacon, Christopher D.},
   author={Mckernan, James},
   title={On Shokurov's rational connectedness conjecture},
   journal={Duke Math. J.},
   volume={138},
   date={2007},
   number={1},
   pages={119--136},
}


\bib{KM98}{book}{
   author={Koll{\'a}r, J{\'a}nos},
   author={Mori, Shigefumi},
   title={Birational geometry of algebraic varieties},
   series={Cambridge Tracts in Mathematics},
   volume={134},
   publisher={Cambridge University Press, Cambridge},
   date={1998},
}



\bib{Nak16}{article}{
   author={Nakamura, Yusuke},
   title={On semi-continuity problems for minimal log discrepancies},
   journal={J. Reine Angew. Math.},
   volume={711},
   date={2016},
   pages={167--187},
}

\bib{NS22}{article}{
   author={Nakamura, Yusuke},
   author={Shibata, Kohsuke},
   title={Inversion of adjunction for quotient singularities},
   journal={Algebr. Geom.},
   volume={9},
   date={2022},
   number={2},
   pages={214--251},
}

\bib{NS25}{article}{
   author={Nakamura, Yusuke},
   author={Shibata, Kohsuke},
   title={Inversion of adjunction for quotient singularities II: non-linear
   actions},
   journal={Algebr. Geom.},
   volume={12},
   date={2025},
   number={4},
   pages={443--496},
}

\bib{NS26}{article}{
   author={Nakamura, Yusuke},
   author={Shibata, Kohsuke},
   title={Inversion of adjunction for quotient singularities III:
   semi-invariant case},
   journal={J. Lond. Math. Soc. (2)},
   volume={113},
   date={2026},
   number={5},
   pages={Paper No. e70545},
%   issn={0024-6107},
%   review={\MR{5067336}},
%   doi={10.1112/jlms.70545},
}

\bib{NS4}{article}{
   author={Nakamura, Yusuke},
   author={Shibata, Kohsuke},
   title={A counterexample to the PIA conjecture for minimal log discrepancies},
   journal={to appear in Duke Math. J.},
   eprint={arXiv:2404.06164v3}
}



\bib{Seb04}{article}{
   author={Sebag, Julien},
   title={Int\'{e}gration motivique sur les sch\'{e}mas formels},
   language={French, with English and French summaries},
   journal={Bull. Soc. Math. France},
   volume={132},
   date={2004},
   number={1},
   pages={1--54},
}

\bib{Sho04}{article}{
   author={Shokurov, V. V.},
   title={Letters of a bi-rationalist. V. Minimal log discrepancies and
   termination of log flips},
   journal={Tr. Mat. Inst. Steklova},
   volume={246},
   date={2004},
   number={Algebr. Geom. Metody, Svyazi i Prilozh.},
   pages={328--351},
   translation={
      journal={Proc. Steklov Inst. Math.},
      date={2004},
      number={3 (246)},
      pages={315--336},
   },
}



\bib{Tak03}{article}{
   author={Takayama, Shigeharu},
   title={Local simple connectedness of resolutions of log-terminal
   singularities},
   journal={Internat. J. Math.},
   volume={14},
   date={2003},
   number={8},
   pages={825--836},
   doi={10.1142/S0129167X0300196X},
}

\bib{Yas16}{article}{
   author={Yasuda, Takehiko},
   title={Wilder McKay correspondences},
   journal={Nagoya Math. J.},
   volume={221},
   date={2016},
   number={1},
   pages={111--164},
%   issn={0027-7630},
%   review={\MR{3508745}},
%   doi={10.1017/nmj.2016.3},
}

\end{biblist*}
\end{bibdiv}
\end{document}